\documentclass[11pt, reqno]{amsart}
\usepackage{amsfonts,amssymb,amsmath,amsthm, mathtools}
\usepackage[normalem]{ulem}
\usepackage[headinclude,DIV13]{typearea}
\usepackage[utf8]{inputenc}
\usepackage{fancyref}
\usepackage{hyperref}
\usepackage{cleveref}
\usepackage{bigints}
\usepackage{algorithm}
\usepackage{epsfig}
\usepackage{algpseudocode}
\algrenewcommand\algorithmicrequire{\textbf{Input:}}
\algrenewcommand\algorithmicensure{\textbf{Output:}}
\usepackage{mathrsfs,bbm,stmaryrd}
\usepackage{tikz,pgfplots}
\usepackage{subfig}
\usetikzlibrary{positioning,patterns}
\usetikzlibrary{decorations.pathreplacing,calligraphy}
\usepackage{geometry}
\usepackage{graphicx, psfrag,enumerate,enumitem}
\usepackage[justification=centering]{caption}
\usepackage{soul}
\usepackage{cancel}
\usepackage{color}
\usepackage{setspace}
\usetikzlibrary{trees}
\usepackage{pgfplots}
\pgfplotsset{width=8cm,compat=1.9}
\newtheorem{theorem}{Theorem}[section]
\newtheorem{proposition}[theorem]{Proposition}

\newtheorem{corollary}[theorem]{Corollary}
\newtheorem{lemma}[theorem]{Lemma}

\newtheorem{assumption}{Assumption}
\newtheorem{rem}[theorem]{Remark}
\numberwithin{equation}{section}

\renewcommand{\P}{\mathbb{P}}

\newcommand{\E}{\mathbb{E}}
\newcommand{\G}{\mathbb{G}}
\newcommand{\N}{\mathbb{N}}
\newcommand{\eps}{\varepsilon}

\newcommand{\Zset}{\mathbb{V}}
\newcommand{\Zbarset}{\overline{\mathbb{V}}}
\newcommand{\Z}{\nu}
\newcommand{\Zbar}{\bar{\Z}}
\newcommand{\Zhat}{\widehat{\Z}}
\newcommand{\Zbarhat}{\overline{\Zhat}}
\newcommand{\Zhatset}{\mathbb{W}}
\newcommand{\Zbarhatset}{\overline{\Zhatset}}

\newcommand{\minus}{\scalebox{0.35}[1.0]{$-$}}
\newcommand{\ssigma}{\boldsymbol{\sigma}}

\newcommand{\e}{\mathfrak{e}}
\newcommand{\ga}{\widehat{\boldsymbol{\gamma}}}

\usepackage[bibstyle=authoryear,
backend=bibtex,
style=numeric,    
url=false, 
doi=false,
eprint=false,
sorting=nty]{biblatex}
\begin{document}
\renewcommand{\labelenumii}{(\roman{enumi}).\alph{enumii}}

\title[Spinal study of a population model for colonial species with interactions]
{Spinal study of a population model for colonial species  with interactions and environmental noise}

\author{Sylvain Billiard}
\address{Sylvain Billiard,  Univ. Lille, CNRS, UMR 8198 – Evo-Eco-Paleo, F-59000 Lille, France.}
\email{sylvain.billiard@univ-lille.fr}

\author{Charles Medous}
\address{\parbox{\linewidth}{Charles Medous, Univ. Grenoble Alpes, CNRS, IF, 38000 Grenoble, France and
Univ. Grenoble Alpes, INRIA, 38000 Grenoble, France}}
\email{charles.medous@univ-grenoble-alpes.fr}

\author{Charline Smadi}
\address{Charline Smadi, Univ. Grenoble Alpes, INRAE, LESSEM, 38000 Grenoble, France
and Univ. Grenoble Alpes, CNRS, Institut Fourier, 38000 Grenoble, France}
\email{charline.smadi@inrae.fr}

\date{}

\maketitle

\begin{abstract}
We introduce and study a stochastic model for the dynamics of colonial species, which reproduce through fission or fragmentation. The fission rate depends on the relative sizes of colonies in the population, and the growth rate of colonies is influenced by intrinsic and environmental stochasticities. Our setting thus captures the effect of an external noise, correlating the trait dynamics of all colonies. In particular, we study the effect of the strength of this correlation on the distribution of resources between colonies. We then extend this model to a large class of structured branching processes with interactions in which the particle type evolves according to a diffusion. The branching and death rates are general functions of the whole population. In this framework, we derive a $\psi$-spine construction and a Many-to-One formula, extending previous works on interacting branching processes. Using this spinal construction, we also propose an alternative simulation method and illustrate its efficiency on the colonial population model. The extended framework we propose can model various ecological systems with interactions, and individual and environmental noises. 
 \end{abstract}

 \vspace{0.2in}

\noindent {\sc Key words and phrases}: fusion-fission population dynamics, interacting populations, diffusion process, spinal constructions, Many-to-One formula, simulation method, environmental noise, ecology

\bigskip

\noindent MSC 2000 subject classifications: 92D25, 60J80, 60J85, 60h10, 60J60, 60K37

\renewcommand{\contentsname}{Table of contents}
\renewcommand\bibname{Bibliography}




\section{Introduction}  \label{section: intro}

\par The life cycle of colonial animal species often includes a mode of reproduction through fission or fragmentation after a colony has grown to a sufficiently large size (e.g. social spiders \cite{lubin80,aviles00,durkin23}, ants \cite{finand2024solitary, peeters2010colonial}, corals \cite{devantier1989,lirman03} or colonial tunicates \cite{ben22}). Many plants (e.g. bryophytes \cite{frey2011asexual}) or microbial species \cite{sauer2022biofilm} show similar life cycles patterns.  Describing the dynamics of such species needs to consider the dynamics of two nested levels: the size of each colony which varies because of growth (ontogeny) and fission, and the number of colonies which is governed by colony birth (fission) or death (local extinction). 
\par Few models have considered the demographic particularities of colonial species, and most often for addressing specific questions. Some studies aimed at evaluating the effect of colonies death and fission on the extinction of populations, for instance with environmental or human perturbations, or explaining the observed structuration in size (or age) of colonies found in natural populations \cite{lirman03,lasker90,nakamaru14,linacre2003demographic}. Such models also generally lack widely shared processes in natural populations such as the heterogeneity between colonies that might impact their death or fission rates, the stochasticity of colonial growth, fission and death, and the effect of colony size on competition.  Our goal in this paper is to provide in a single framework a general stochastic model describing the dynamics of a colonial species including all these mechanisms.
\par In our model, we consider a colony as an ‘individual’, itself composed of sub-units such as ‘individual’ cells, branches, ants, spiders or even birds. The trait we consider is associated to the  colony and not to its sub-units. This trait can represent the size or volume of the colony, that is related to the number of sub-units, e.g. the number of workers or reproductive queens in an ant colony. The dynamics of the population of colonies is governed by three mechanisms. First, we suppose that the size of a colony can vary during its lifetime because of intrinsic factors (stochastic birth and death of sub-units) or extrinsic factors (environmental variability) \cite{ben22,boulay10,chen22,lubin80}. This is well represented by the evolution of the traits following a stochastic differential equation (SDE) with intrinsic and environmental components of the fluctuations. Second, the establishment of new colonies follows fission events. Such fissions can be associated to reproductive events such as in some ant species where queens can found new nests accompanied by a variable number of workers coming from the ancestral colonies \cite{finand2024solitary}. Fissions can also be due to fortuitous random events such as when corals are fragmented or spider webs are split in daughter webs \cite{lubin80,durkin23,aviles00}. In any case, the newly established colonies can randomly vary in their initial size, a phenomenon captured by our model as when a birth occurs, the size (or trait) of the colony is split into two random size daughter colonies. Third, we consider competitive interactions between colonies as affecting the fission rate. The fission rate is supposed to be proportional to the relative size of a colony compared to the sum of all other colonies' size. This can represent colonies competing for the resource of their environment as foraging is directly related to the size of the colony, for example the number of foraging ants, the volume of water filtered by branches of coral, the surface of spider webs \cite{tschinkel95,shik08,finand2024solitary,yoshioka98,chen22,aviles98}. To sum up, our model allows to study the population dynamics of colonial species accounting for within and between colonies heterogeneity, in a time varying environment, with competition. We will in particular study the relative roles of the different sources of variability on the maintenance and composition of the population.

More specifically, we consider a structured birth and death process of a population of colonies each characterized by a trait $x$, e.g. their size. We suppose that competition between colonies affects their fission rate such as it is proportional to their relative size $x / R_t$ where $R_t$ is the sum of the sizes of all colonies. At a fission event of a colony of size $x$, resources are shared by the two daughter colonies. One colony gets a fraction $\Theta$ of the size and the other a fraction $(1-\Theta)$ with $\Theta$ a random variable on $(0,1)$.
Between two branching events, the colony sizes evolve according to a diffusion; if at time $t$, $N_t$ denotes the number of colonies in the whole population, $u$ the label of a single colony and $X_t^u$ its size,
\begin{align}\label{eq: eds indiv}
\text{d}X^u_t := aX^u_t\text{d}t + \frac{\sigma}{\sqrt{1+\delta^2}} X^u_t(\text{d}B^u_t + \delta\text{d}W_t)
\end{align}
where $a>0$ is the growth rate of the colony, $\sigma > 0$ is the intensity of the total noise, $(B^u_t, u\in\mathbb{G}(t))$ denotes a $N_t$-dimensional intrinsic Wiener process and $W_t$ an independent extrinsic Wiener process. The parameter $\delta>0$ quantifies the relative intensity of the extrinsic and intrinsic noises. The processes $B^v$ for $v$ describing the set of colonies alive are independent and all independent of $W$. The latter process corresponds to the environment and is common to all colonies alive in the same environment. Note that the value of the correlation factor $\delta$ does not change the law of a colony trait, but the correlation between the dynamics of traits of the different colonies of the population. Depending on $\delta >0$ the noise due to the common environment may be more or less important compared to the independent intrinsic noises in the colonies. 

We establish results about the long term behavior of those interacting colonies, using spine techniques.
Spine techniques have been developed to study Branching diffusions since the work of Chauvin and Rouault \cite{Chauvin88} who established limit theorems for the law of the position of the right-utmost particle. Such methods are based on a Girsanov change of probability measure associated with a distinguished individual called the spine. 
Branching Brownian motion with constant division rate is related to the solution of Kolmogorov–Petrovskii–Piskounov (KPP) diffusion equations \cite{mckean75} and spine techniques have been used to establish the existence, asymptotics, and uniqueness of traveling wave solutions of the KPP equation \cite{harris99, kyprianou04}. We also refer to Hardy and Harris \cite{hardy06} for an intuitive, spine-based proof of a large deviations result in non-structured branching Brownian motion. This framework has been extended to structured branching diffusion, where the trait follows an Ornstein–Uhlenbeck process, with results that include the long-term behavior of the system, such as the speed of colonization in space and type dimensions and population growth rates within this asymptotic shape \cite{git07}.
The general case of structured branching diffusion has been studied by Hardy and Harris \cite{Hardy09}, who used this spinal decomposition to prove the $L^p$-convergence of some key martingales, which was later used to establish strong laws of large numbers \cite{englander2010strong}. 

We refer to Bansaye and coauthors \cite{BT11,bansaye2013extinction} for a spine-based asymptotic study of the behavior of parasite infection in cells, which follows a Feller diffusion with trait-dependent multiplicative jumps. Later, Cloez \cite{Cloez17} studied a structured branching process where the individual traits follow a general Markov process and determined the asymptotic behavior of the particle system with spine methods. In a similar framework, we also refer to Marguet \cite{marguet19}, who studied a sampled individual in the asymptotic population.

The martingale change of measure associated with the spinal construction involves either an eigenfunction $\psi$ of some operator $\mathcal{G}$ related to the branching mechanisms \cite{biggins2004measure, englander2007branching, Cloez17}, or a Feynman-Kac path-integral formula \cite{harris96, marguet19, B21, medous23}.
In the case of Galton-Watson trees and non-structured branching processes, the constant function $\psi \equiv 1$ is an eigenfunction of the operator $\mathcal{G}$. For structured branching diffusion, Cloez \cite{Cloez17} exhibited eigenfunctions of $\mathcal{G}$ for various branching mechanisms. In the case of branching processes with interactions, the interplay between the individuals is more complex, and it is arduous to find an explicit eigenfunction. Bansaye \cite{B21} established that, for a non-structured branching process with size-dependent interactions, the function $\psi = 1/N_t$ where $N_t$ denotes the number of individuals in the population is always an eigenfunction of $\mathcal{G}$. However, incorporating a trait dependency into the branching mechanisms compromises this property, resulting in the absence of general eigenfunctions \cite{bansaye22}. Our motivating model of colonial species does not fit within the framework of the aforementioned studies, which must be extended to encompass a broader class of structured branching processes with interactions.

The article is organized as in the following. We first introduce our stochastic model of a population of a colonial species. We then establish a spinal construction of the latter, that enables us to study long-term properties of the population, thanks to the use of Many-to-One formulas and the study of the (simpler) spinal process. In particular we show that if the intrinsic variability is large while the environmental variability is of a smaller order, most of the resources are held by a unique colony. Otherwise, the resources are more evenly shared among colonies. We then use the spine process to propose an alternative simulation method of this colonial dynamics, that we compare to the classical method of birth and death process simulation.
Finally we extend the generality of this method by derivating a spine construction with a general $\psi$ function- that will be referred to as a $\psi$-spine construction- in the context of branching diffusion with interactions. In this extended setting, we consider a wide class of continuous-time structured branching diffusions with general interactions where each individual (instead of 'colonies' in the extended model to be as general as possible) is characterized by its trait, taking values in a subset of $\mathbb{R}^d$. The lifespan of each individual is exponentially distributed with a time-inhomogeneous rate that depends on the traits of all individuals. Upon an individual's death, a random number of children are generated, each inheriting random traits at birth. Between these branching events, the evolution of the traits of all individuals in the population follows a $N_t$-dimensional SDE where $N_t$ is the population size. This SDE depends on the traits of all the individuals in the population, extending previous works \cite{B21, medous23} where trait dynamics were deterministic, to include many physical models based on Brownian motion, see \textit{e.g.} \cite{damos11} and the references therein.
In this general framework, Many-to-One formulas linking the behavior of the original and the spinal processes thus appear as a promising tool to model and study various ecological systems that include interactions, as well as individual and environmental noises.

\section{A colonial population correlated through a common environment.}\label{section: example}

\subsection{Description of the model}
We consider a population $\nu_t$ at time $t$, composed of $N_t \geq 0$ colonies. 
Each colony is associated to a unique label $u\in\mathcal{U}$, where the set $\mathcal{U}$ is defined according to the Ulam-Harris-Neveu notation:
$$\mathcal{U} := \left\{\emptyset\right\} \cup \bigcup_{k \geq 0}\left(\mathbb{N}^*\right)^{k+1}.$$
The set of labels of the colonies present in the population at time $t$ is denoted $\mathbb{G}(t) \subset\mathcal{U}$. We denote $X^u_t \in \mathbb{R}_+^*$ the trait of a colony labeled by $u \in \mathbb{G}\left( t\right)$. $X^u_t$ can represent the mass or size of colony $u$, or its total quantity of resources.
Using these notations, we can write 
\begin{equation*}
\Z_t := \sum_{u \in \mathbb{G}\left( t\right)}\delta_{X^u_t} \in \mathbb{U}
\end{equation*}
where $\mathbb{U}$ denotes the set composed of all finite point measures on $\mathbb{R}_+^*$ defined as
\begin{equation*}
\mathbb{U} := \Big\{\sum_{i=1}^N \delta_{x^i},\ N \in \mathbb{N}, \ \left(x^i, 1 \leq i \leq N\right) \in  \left(\mathbb{R}_+^*\right)^{N} \Big\}.
\end{equation*}
We consider a dyadic birth and death process such that a colony of resource $x \in \mathbb{R}_+^*$ dies at rate $\mu$ and splits at rate $\lambda x / R_t$ where $R_t := \langle\nu_t,I_d\rangle$ denotes the total amount of resources in the population and $\lambda$ is a positive parameter. 
When a colony of size $x$ splits, the newly founded colony inherits $\Theta x$ resources from the parent colony, while the latter keeps $(1-\Theta)x$ resources, with $\Theta$ a random variable on $[0,1]$ of probability measure $q$.
Between two branching events at times $T_i$ and $T_{i+1}$, the resources evolve according to a SDE \eqref{eq: eds indiv}.

One of the interesting features of the model we have just
introduced is that we take into account the competition for space
or resources in the reproduction rate. This
is technically more difficult to handle than a model taking
competition into account in the death term, (which corresponds to the most usual choice, the prototypical
example being the competitive Lotka-Volterra model \cite[Chapter 7.8]{allen2010introduction}) but is
much more realistic for many species (for instance the ones in which larvae compete for resources and have little chance of reaching maturity).
A second interesting feature of the model is that it allows taking into account any law of resource sharing between two offspring colonies at a fission event. Recent works \cite{marguet2024parasite}, considering structured branching processes without interaction, have shown that it has a large impact on the long time behavior of the population. We will present some results in the same direction.
Finally, we consider noise arising from two distinct sources: independent internal sources within individual colonies and a shared environmental source affecting all colonies.
The question of interaction and relative effects of intrinsic and environmental stochasticities is a major question in ecology \cite{lande1998demographic,bonsall2004demographic,wilmers2007perfect}.
By letting vary the parameter $\delta$ we can modify the level of correlation of the noises impacting the dynamics of the colonies, and study the effect of this correlation.

\subsection{Spinal Construction}\label{section: spine construct ex}
In this section, we introduce the spine process $(\bar{\nu}_t, t\geq 0)$ associated with the population process and give the relationship between the law of this spine process (that does not describe any real population) and the law of the population process. In the spine process, we will also refer to the elements in the population as colonies described by a trait in $\mathbb{R}_+^*$.
The associated spinal process $((E_t,\Zbarhat_t),t\geq 0)$ distinguishes a particular colony -the spine- of label $E_t$ from the set of labels $\widehat{\mathbb{G}}(t)$ in the spinal population $\Zbarhat_t$. In order to have a simple distinctive feature between both processes, we will denote by $\widehat{\nu}_t$ the marginal spinal population, and put a hat on each value that refers to the spinal process. At all time, we will denote by $\widehat{Y}_t$ the trait of the spine $E_t$ and by $\widehat{X}^u_t$ the trait of the colony of label $u$ belonging to $\widehat{\mathbb{G}}(t)\setminus \{E_t\}$ in the spinal population.
The SDE describing the evolution of the traits between two branching events involves an additional drift term compared to the resources dynamics in the biological process:
\begin{equation}\label{eq: eds hors spine}
\text{d}\widehat{X}^u_t := \left(a + \sigma^2 \frac{\delta^2}{1 + \delta^2} \right)\widehat{X}^u_t\text{d}t + \frac{\sigma}{\sqrt{1+\delta^2}} \widehat{X}^u_t\left(\text{d}\widehat{B}^u_t + \delta\text{d}\widehat{W}_t\right), \quad \text{for all } u\in\widehat{\mathbb{G}}(t)\backslash\{E_t\},
\end{equation}
and
\begin{equation}\label{eq: eds spine}
\text{d}\widehat{Y}_t := \left(a+\sigma^2\right)\widehat{Y}_t\text{d}t +  \frac{\sigma}{\sqrt{1+\delta^2}} \widehat{Y}_t\left(\text{d}\widehat{B}^*_t + \delta\text{d}\widehat{W}_t\right)
\end{equation}
where $\widehat{B}^*, \widehat{B}^u, \widehat{W}$ are independant Wiener processes.
The initial spine is chosen randomly in the initial population $\widehat{\nu}_0$ among the initial labels $\widehat{\mathbb{G}}(0) := \{1, \cdots, \langle\widehat{\nu}_0,1\rangle\}$, with a probability proportional to its trait. That is, for all $j \in \widehat{\mathbb{G}}(0)$, 
\begin{equation}\label{eq: init spine}
\mathbb{P}\left(E_0=j\right) = {X^j_0}/{\langle\widehat{\nu}_0,I_d\rangle}.
\end{equation}
Each colony of trait $x$ in the spine process divides into two children at a rate $\lambda x / \widehat{R}_t$, where $\widehat{R}_t:= \langle\widehat{\nu}_t,I_d\rangle $ is the sum of the traits of all colonies in the spine process. Any non-spinal colony dies at rate $\mu$ while the spine does not go extinct. If the spine divides at time $t$, the new spine is chosen between the two children with probability proportional to their traits, such that
\begin{equation}\label{eq: next spine choice}
\mathbb{P}\left(E_t=E_{t-}i\right) = \frac{\widehat{X}^{E_{t-}i}}{\widehat{X}^{E_{t-}1} + \widehat{X}^{E_{t-}2}}, \quad \text{for } i \in \{1,2\}.
\end{equation}

We can now introduce the theorem that links the law of the non-extinct biological process of interest, that properly describes the colonies dynamics, and the law of the spine process. 
We denote by $\widehat{N}_t$ the number of colonies at time $t\geq 0$ in the spinal process.
\begin{theorem}\label{thm: sampling colony}
    For $U_t$ a colony picked uniformly at random in $\mathbb{G}(t)$ and for every measurable function $F$ on the set of càdlàg trajectories on $\mathbb{R}_+^*\times\mathbb{U}$, we have
\begin{equation}\label{eq: sampling}
    \mathbb{E}_{\nu_0}\left[\mathbbm{1}_{\left\{\mathbb{G}(t) \neq \emptyset\right\}}F\left(X^{U_t}_t,(\Z_s,s\leq t)\right) \right] =  R_0e^{(a-\mu)t} \mathbb{E}_{\nu_0}\left[\frac{1}{\widehat{Y}_t\widehat{N}_t} F\left(\widehat{Y}_t,(\Zhat_s, s\leq t)\right)\right].
\end{equation}
\end{theorem}
Notice that this equality in law holds under the event of survival of the population. 
The Many-to-One formula associated to this $\psi$-spine construction can be deduced from Theorem \ref{thm: sampling colony} by taking the measurable function $F$ defined on the set of càdlàg trajectories on $\mathbb{R}_+^*\times\mathbb{U}$ as
    $$
    F\left(X^{U_t}_t,(\Z_s,s\leq t)\right) := f\left(\left(X^{U_t}_s, \nu_s\right),s\leq t\right)\widehat{N}_t.
    $$
\begin{corollary}\label{cor: mto colony}
    For every measurable function $f$ on the set of càdlàg trajectories on $\mathbb{R}_+^*\times\mathbb{U}$:
\begin{equation}\label{eq: mto}
    \mathbb{E}_{\nu_0}\left[\mathbbm{1}_{\left\{\mathbb{G}(t) \neq \emptyset\right\}}\sum_{u\in\mathbb{G}(t)}f\left(\left(X^u_s, \nu_s\right),s\leq t\right) \right] =  R_0e^{(a-\mu)t} \mathbb{E}_{\nu_0}\left[ \frac{f\left(\left(\widehat{Y}_s, \widehat{\nu}_s\right), s\leq t\right)}{\widehat{Y}_t} \right].    
\end{equation}
\end{corollary}

Corollary \ref{cor: mto colony} is a key result, at the basis of the next two sections.
It allows us to derive some results on the long-time behavior of the population stated in Section \ref{section: long-time}.
Besides, from Corollary \ref{cor: mto colony}, we can deduce an algorithmic method to efficiently estimate some quantities over the population (Section \ref{section_simu}).

\subsection{Long-time behavior of the population}\label{section: long-time}
Formula \eqref{eq: mto} allows us to link the expectation of key quantities of the dynamics of the process to expectations related to the spinal process. Let us first focus on quantities describing the global dynamics of the population process, that is the population size and the total population amount of resources.
\begin{proposition}\label{prop: basic qts}
The population size $N_t$ follows the law of an $M/M/\infty$ queue with arrival rate $\lambda$ and service rate $\mu$. The mean extinction time is 
\begin{equation} \label{mean_ext_time}
    \E_{\Z_0}[T_0] = \sum_{k = 0}^{N_0-1}\frac{k!}{\lambda(\lambda/\mu)^k}\sum_{j \geq k+1}\frac{(\lambda/\mu)^j}{j!}
\end{equation}
and the expected size satisfies
\begin{equation}\label{eq: mean pop size}
 \lambda te^{-\mu t} + N_0e^{-\mu t} \leq \mathbb{E}_{\nu_0}\left[ N_t\right] \leq f(t) + N_0e^{-\mu t}
\end{equation}
where the function $f$ satisfies $f(t) = O(1/t)$ as $t$ goes to infinity.
\end{proposition}

\begin{rem}
An $M/M/\infty$ queue with arrival rate $\lambda$ and service rate $\mu$ is a time changed logistic process of birth rate $\lambda$, death rate $0$ and competition parameter $\mu$. Indeed, if the population is in state $N$, the next state of the population is $N+1$ with probability 
${\lambda}/{\mu N}= {\lambda N}/{\mu N^2}$
for both $M/M/\infty$ queue and logistic process. Since the jump rates are higher in the case of the logistic process (multiplied by a factor $N$ when the population is in state $N$), we can deduce that the $M/M/\infty$ queue has an extinction time larger that the one of the mentioned logistic process.
\end{rem}

A first consequence of Proposition \ref{prop: basic qts} is that the law of $N_t$ only depends on the parameters $\lambda$ and $\mu$, and not on the noise parameters of the dynamics, $\sigma$ and $\delta$. This contrasts with classical results on the role of noise on the population persistence, which conclude that a higher noise increases the population death rate (see \cite{conner99} for instance). The main difference of our model with cited works is that there is a strong interaction between colonies to have access to the resource, which translates into a competition in the birth rate. This fact illustrates that taking into account interactions between colonies has a major effect on the population dynamics, and that making the approximation that colonies behave independently may lead to erroneous conclusions.
\begin{figure}[h!]
    \centering
    \includegraphics[width=0.97\linewidth]{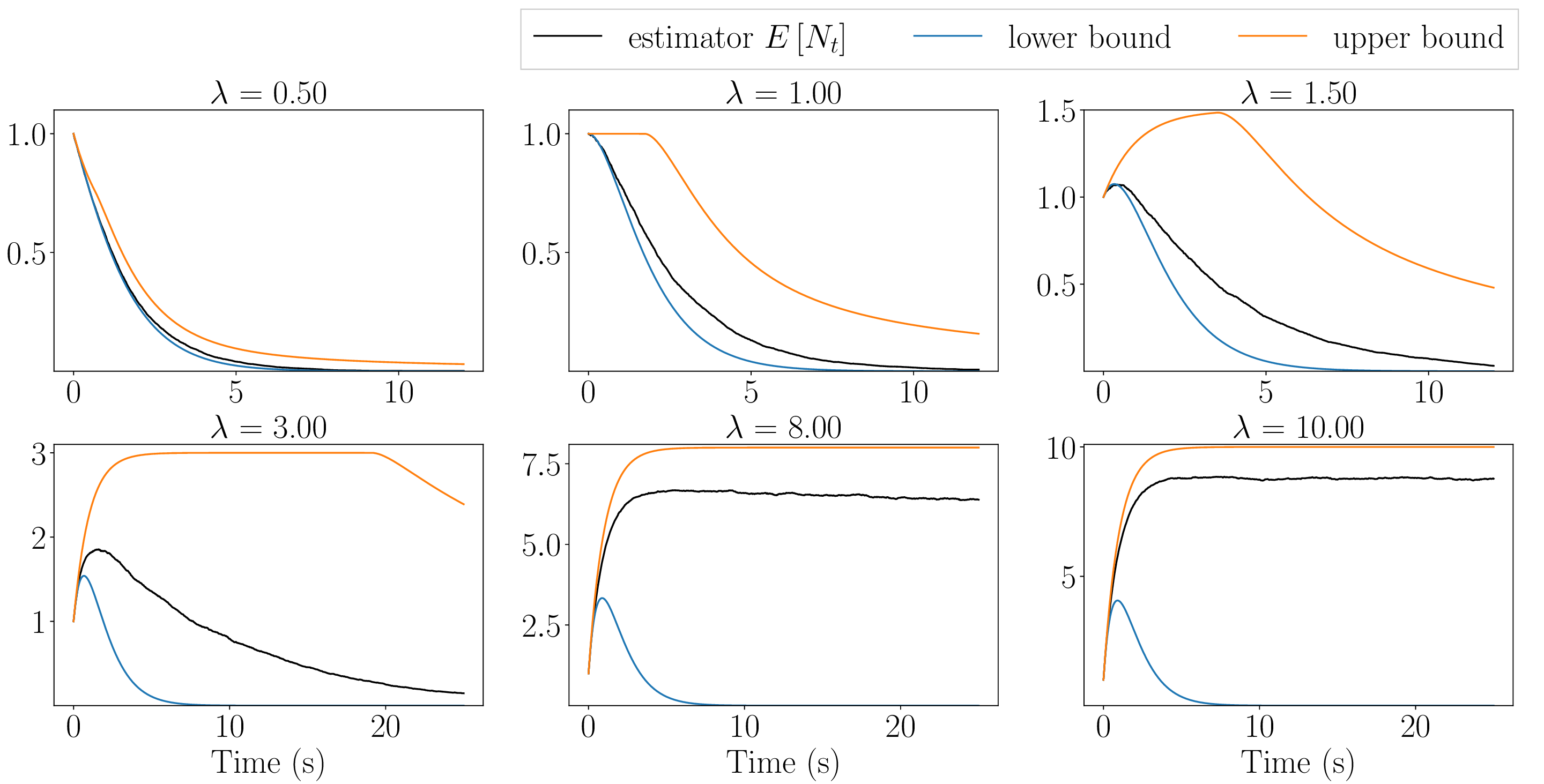}
    \caption{Monte-Carlo estimation of $\E[N_t]$ using $M = 10^4$ trajectories, for different values of the division rate $\lambda$. Other parameters are $\mu=1$ and $N_0= 1$.}
    \label{fig: simu E N_t}
\end{figure}

From \eqref{eq: mean pop size}, we see that the mean population size is bounded by a quantity that decreases as $1/t$ 
for large $t$ whereas the lower bound decreases to $0$ exponentially fast. In fact, we cannot find more precise boundaries without disjointing cases, because all the asymptotics between the two bounds are possible, as illustrated by Figure \ref{fig: simu E N_t}. This depends on the value of $\lambda$ compared to $\mu$. 

\begin{proposition} \label{prop:var}
The expectation of the resources $R_t$ in the original population verifies
    \begin{equation}\label{eq: basic qts}
    \mathbb{E}_{\nu_0}\left[ R_t\right] = R_0e^{(a-\mu)t}.
\end{equation}
Besides, the variance of the total resource in the population satisfies
\begin{eqnarray*}
        \frac{\mu +\frac{\sigma^2}{1 + \delta^2}}{R_0\left(\mu + \frac{\sigma^2}{1 + \delta^2} - 2 \lambda \E[\Theta(1-\Theta)]\right)}\left(e^{\left( \mu +\sigma^2 -  2 \lambda \E[\Theta(1-\Theta)] \right)t} - e^{\frac{\sigma^2\delta^2}{1 + \delta^2}t} \right) + e^{\frac{\sigma^2\delta^2}{1 + \delta^2}t} \\
        \leq  \frac{\mathbb{V}ar_{\nu_0}\left[R_t \right]}{R_0^2e^{2(a-\mu)t}}+1 \leq \frac{1}{R_0}e^{\left( \mu +\sigma^2 \right)t} + \frac{R_0 -1}{R_0}e^{\frac{\sigma^2\delta^2}{1 + \delta^2} t}.
\end{eqnarray*}
\end{proposition}
Notice that the lower bound is a continuous function of the parameters and the case $\mu + \sigma^2/(1 + \delta^2) = 2 \lambda \E[\Theta(1-\Theta)]$ is still well defined.

\begin{figure}[h!]
    \centering
    \includegraphics[width=0.97\linewidth]{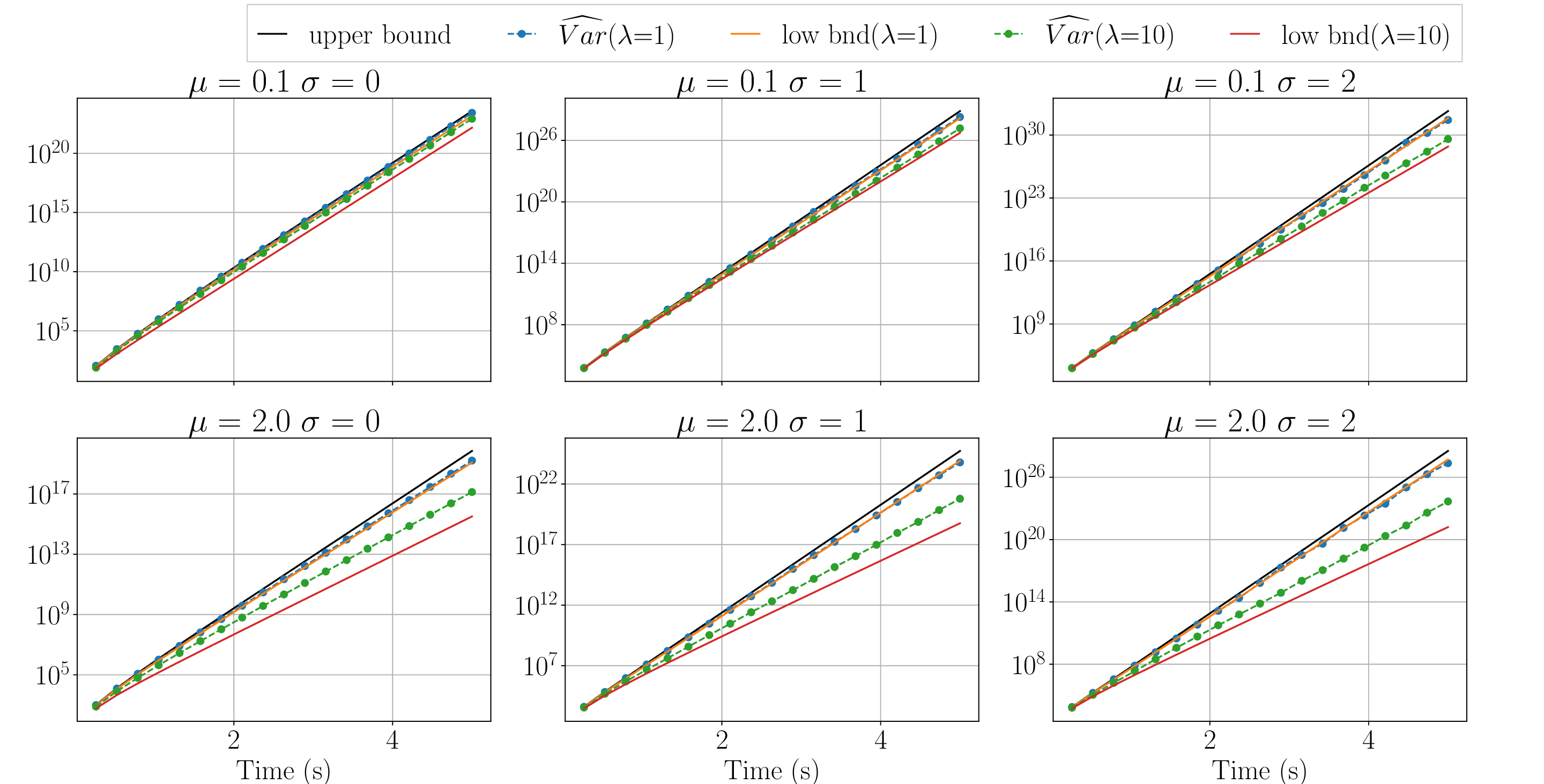}
    \caption{Monte-Carlo estimation of $\mathbb{V}ar[R_t]$ using $M = 10^6$ trajectories, for different values of $\lambda, \mu $ and $\sigma$. Other parameters are $R_0= 15, a= 5$ and $\delta = 1$.}
    \label{fig:var mieux}
\end{figure}

Equation \eqref{eq: basic qts} shows that the average resource in the population only depends on its Malthusian parameter $a$ and the individual death rate $\mu$. It is not affected by the strength of the noise.
Its variance, however, strongly depends on the strength of the noise (parameter $\sigma$), as well as the correlation between individual stochasticities (parameter $\delta$).
As we can see in Figure \ref{fig:var mieux}, the bounds found in Proposition \ref{prop:var} are a good approximation of the variance.
From the simulations (we observe the same trend for both lower and upper bounds), we see that the variance of $R$ increases with $\sigma$ and decreases with $\mu$ as we would expect. Notice that the lower bound in Proposition \ref{prop:var} decreases with $\lambda$ while the upper bound does not depend on the fission rate. The simulations in Figure \ref{fig:var mieux} show that for small division rates (the blue dots with $\lambda = 1$), both boundaries are close, while when $\lambda$ is higher (the green dots with $\lambda = 10$), these bounds are less optimal and the variance decreases. \\

Recall that the resources do not impact the branching mechanisms at the population level, while it does at an individual level. Hence the quantity of interest for studying the behavior of the colonies in the population is not their traits but the fraction of resources $x/R$ that they own. This is the content of the next result, which focuses on the distribution of these resources among the various colonies in the population.

\begin{proposition}\label{prop: sharing resources}
The average resource owned by a colony picked uniformly at random at a time $t$ satisfies, as $t$ goes to infinity
$$
\E_{\nu_0}\left[\mathbbm{1}_{\left\{\mathbb{G}(t) \neq \emptyset\right\}}X^{U_t}_t\right] \sim  R_0e^{(a-\mu)t} \frac{\mu}{\lambda}\left(1 - e^{-\lambda/\mu}\right).
$$
Depending on the noise intensity and correlation, the proportion of resources owned by the most successful colonies may vary. Let $\mu$, $\lambda$ and $q$ (the distribution of $\Theta$) be fixed. Then: 
\begin{itemize}
    \item If $\sigma^2/(1+\delta^2) \to \infty$, for all $\eps>0$
    \begin{equation}\label{tout_ds1}
        \mathbb{E}_{\nu_0}\Bigg[ \sum_{u \in \mathbb{G}(t)} X^u_t \mathbbm{1}_{\left\{\frac{X^u_t}{R_t} \geq 1-\eps \right\}} \Bigg]\sim \E_{\nu_0}[R_t], \quad (t \to \infty).
    \end{equation}
     \item If  $\sigma^2/(1+\delta^2) \to 0$, for all $1/2<\gamma <1$,
      \begin{equation}\label{pastout_ds1bis}  \limsup_{t \to \infty}\frac{\mathbb{E}_{\nu_0}\Big[ \sum_{u \in \mathbb{G}(t)}X^u_t \mathbbm{1}_{\left\{\frac{X^u_t}{R_t} \geq \gamma \right\}} \Big]}{\E_{\nu_0}[R_t]} \leq \frac{\sqrt{1+8 \lambda \E[\Theta(1-\Theta)]/\mu}-1}{4 \gamma \lambda \E[\Theta(1-\Theta)]/\mu} \underset{\lambda \E[\Theta(1-\Theta)]/\mu \to \infty}{\to} 0 .  \end{equation}
\end{itemize}
We can also obtain a lower bound on the maximal quantity of resources owned by a colony
\begin{equation}\label{tout_ds1bis}
     \liminf_{t \to \infty}\frac{\mathbb{E}_{\nu_0}\left[ \sup_{u \in \mathbb{G}(t)} X^u_t \right]}{\E_{\nu_0}[R_t]} \geq   \left(\frac{\beta+ \beta\frac{\mu}{\lambda} \left(1- e^{-\lambda/\mu}\right)+2}{2\beta- \frac{4\lambda}{\mu}\E[\Theta \ln\Theta]+ 2}\right) \vee  \frac{\mu}{\lambda}\left(1- e^{-\lambda/\mu}\right),
\end{equation}
where
\begin{equation}\label{def: alpha beta}
    \alpha:= {\lambda \E[\Theta(1-\Theta)]}/{\mu} \quad \text{and}\quad \beta := {\sigma^2}/({\mu\left(1+\delta^2\right))}. 
\end{equation}
\end{proposition}

The distribution of resources among colonies is thus strongly affected by the intensity of the fluctuations and their correlation (i.e. the $\sigma$ and $\delta$ parameters respectively).
Equation \eqref{tout_ds1} proves that if the total noise intensity $\sigma$ is large and the fluctuations are not too correlated ($\delta$ of a smaller order than $\sigma$), most of the resources are held by one colony. Equation \eqref{pastout_ds1bis} shows that whatever the noise intensity, if fluctuations are correlated, resources are more evenly shared among several colonies. Indeed, the upper bound is strictly smaller than $1$ if $\gamma$ is close enough to $1$. 
Equation \eqref{pastout_ds1bis} also shows that if the variance $\sigma$ of the noise in the individual traits is too low, the impact of $\delta$ is not significant. This is natural as $\delta$ represents the relative intensity of environmental noise compared to intrinsic noise. Consequently, if the total variance of the noise in trait dynamics is too small, two populations with different environmental sensitivities would still exhibit similar behaviors.
 In addition, when the fission rate is large, the impact of the environmental noise decreases as the population size is large. However, at a fixed fission rate, an increase in the death rate also decreases the effect of environmental noise. This is due to individuals dying too fast such as noise does not have a significant impact on their trajectories.
This result illustrates the importance of distinguishing between intrinsic and environmental stochasticities to predict the long time behavior of a population. Figure \ref{fig:distribution} shows the sharing of resources between individuals, estimated with Monte-Carlo methods, at different times. The stationary distribution is reached in a short time scale independently of the parameters values. The transition between a unique 'winning' colony and a more homogeneous distribution of resources among colonies is continuous with the value of $\sigma^2/(1+\delta^2)$.     
However, when the individual traits are noise-sensitive, the distribution of resources colonies in two populations with different environmental noise levels is distinct. 
\begin{figure}[h!]
    \centering
    \includegraphics[width=\linewidth]{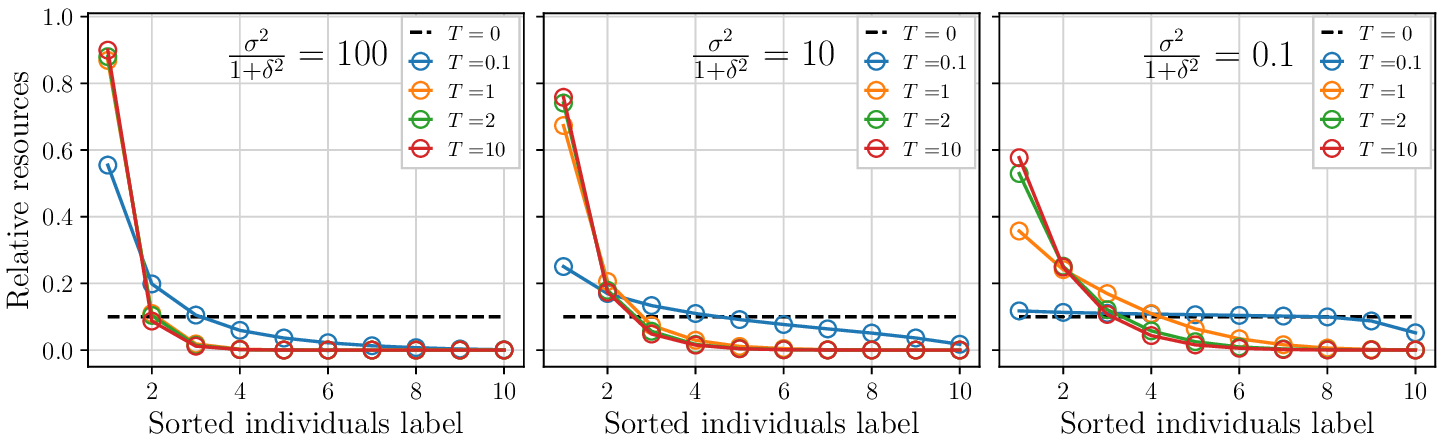}
    \caption{Comparison of the mean sharing of resources at different times $T$ and different values of $\sigma^2/(1+\delta^2)$. The x axis gives the label of the individuals, sorted by trait in a decreasing order. The other parameters are $\lambda = 3, \mu = 1, a = 1, N_0 = 10$ and $R_0 = 10$.}
    \label{fig:distribution}
\end{figure}

\subsection{Spinal algorithm} \label{section_simu}
As illustrated in Section \ref{section: long-time}, the spine process is a powerful tool to derive analytical results on key features of the population, that would otherwise be intractable. It can also be of great interest to consider spine processes when it comes to simulations of the population. As we have emphasized, a trajectory of the spine process does not describe a potential behavior of the colonial population, but Theorem \ref{thm: sampling colony} gives the right correction to apply to estimate expectations from spinal trajectories, using Monte-Carlo methods. This alternative simulation method classifies as an importance sampling method, see \textit{e.g.} \cite{smith1997quick}, as we use samples generated from a different distribution (the spinal one) to evaluate properties on the distribution of interest. The spinal Algorithm \ref{algo: schematic} gives the general method to construct a trajectory of the spine process, and the implemented methods in Python can be found in \cite{github2}. It is similar to the classical acceptance-rejection method, see \textit{e.g.} \cite{champ07} that can be used to draw trajectories of the colonial process, with some extra steps to handle the particular behavior of the spinal colony.
\begin{algorithm}
\caption{\textbf{Spinal trajectory simulation method on $[0,T]$}}\label{algo: schematic}
\begin{algorithmic}
\State Initialize time $t=0$ and choose initial spine $E_0$, according to \eqref{eq: init spine}.
\State \textbf{While} $t < T$:
	\State $\qquad \boldsymbol{1} \longrightarrow$ Propose the next branching time $t_{\text{next}}$. 
	\State $\qquad \boldsymbol{2}\longrightarrow$ Draw the colonies trajectories on $[t,t_{\text{next}}\wedge T)$ from \eqref{eq: eds hors spine} and \eqref{eq: eds spine}.
	\State $\qquad \boldsymbol{3}\longrightarrow$ Decide the type of branching event with rejection method.
	\State $\qquad$\textbf{If} Spinal branching event:
    \State $\qquad \qquad$ Choose new spine according to \eqref{eq: next spine choice}.
    \State $\qquad \boldsymbol{4}\longrightarrow$ Update current time $t \leftarrow t_{\text{next}} \wedge T$
\end{algorithmic}
\end{algorithm}

According to Algorithm \ref{algo: schematic}, generating a  trajectory of the spine process takes a short additional amount of time compared to the direct method, using trajectories of the colonial process. However, it can be faster to use this method in two cases: when the variance of the estimated quantity is smaller in the spine process, and when the computations are easier in the spine process. The first point is illustrated in Figure \ref{fig: compare times}A, where the goal is to estimate the variance of the total resources $R_t$ in the colonial process by using trajectories of $R_t^2$. In this case, the variance in the Monte-Carlo method is of the order $\E[R_t]^4$.
However, from Corollary \ref{cor: mto colony} with $f(x) = x^2$, we see that this value can be estimated from trajectories of $\widehat{R}_t$, thus the variance in the Monte-Carlo method with spinal trajectories is of the order $\E[R_t]^2$.
As a result, the evaluated quantity  converges faster using spinal than colonial trajectories, and a smaller number of iterations $M$ is required from the spine process to reach the same precision. Overall, the spinal method converges in this case approximately 10 times faster than the direct method, despite the fact that the generation of one trajectory of the spine process takes a longer time.
Another interesting point in this particular case is the fact that the trajectories of $R_t^2$ generated from the colonial process can reach really high values when the intrinsic noise $\sigma$ increases. In some trajectories, the highest possible value that can be stored in simple precision (approximately $10^{38}$, according to the norm IEEE 754) is exceeded, resulting in an underestimation of the real value, as we can see in Figure \ref{fig: oversized}. A simple modification consists in using double precision floating numbers, that can reach up to $10^{308}$, but this will drastically increase both the computation time and the memory usage of the algorithm. The spinal method, on the other hand, uses trajectories of $\widehat{R}_t$, that have a significantly lower probability to reach such high values.
\begin{figure}[h!]
    \centering
    \includegraphics[width=0.97\linewidth]{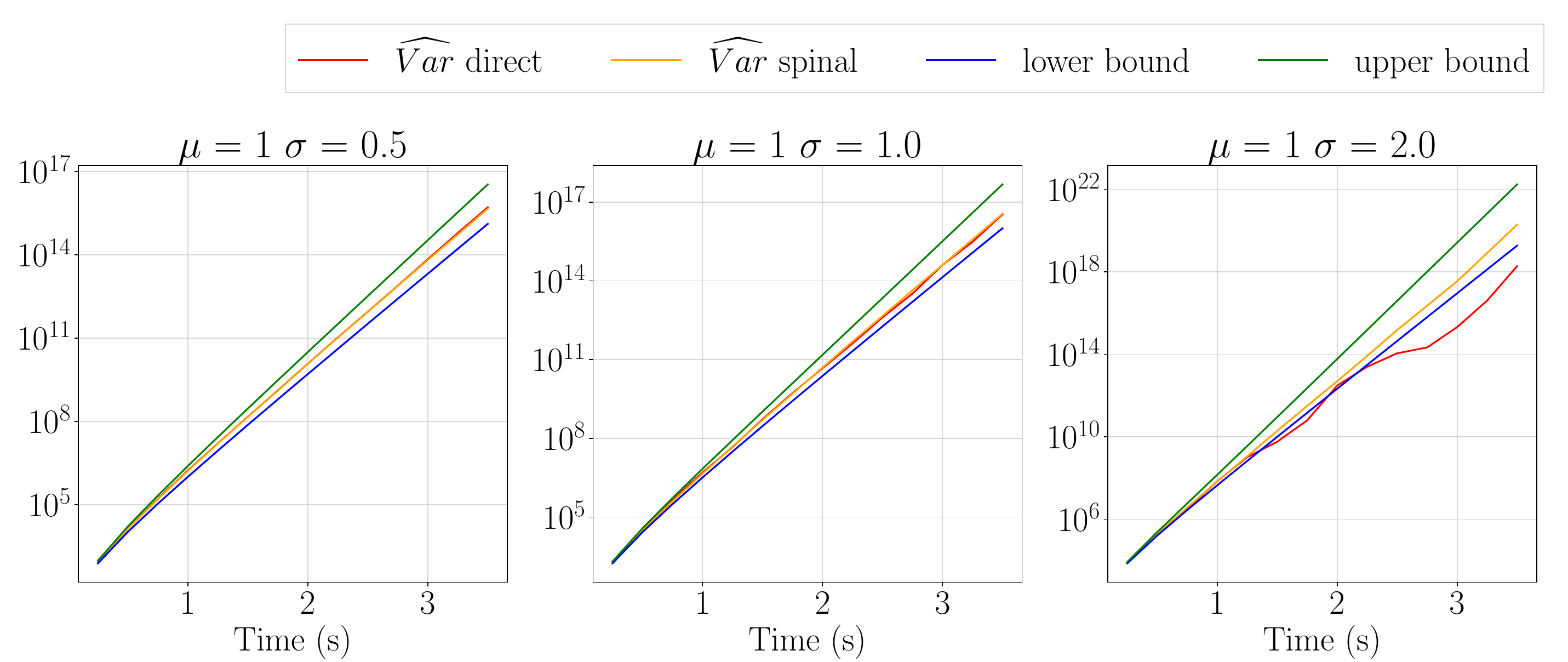}
    \caption{Monte-Carlo estimations of $\mathbb{V}ar[R_t]$ with $10^6$ trajectories from the direct and spinal methods. Parameter values are $a = 5$, $\lambda = 6$, $\delta = 1$, $R_0 = 30$. }
    \label{fig: oversized}
\end{figure}

The second general case for which the spine method can be significantly faster than the direct one is when the estimated quantity needs heavy computations to generate trajectories of the colonial process, while the required trajectories from the spine process are much easier to draw. In Fig.\ref{fig: compare times}B, we estimated the mean size of a colony picked uniformly at random in the colonial population $\mathbb{E}[\mathbbm{1}_{\G(t)\neq 0}X^{U_t}]$. Direct sampling trajectories need to store the vector of traits of all colonies and pick uniformly one colony in this array, while the spine method only necessitates to sample trajectories of $1/\widehat{N}_t$ (Th.\ref{thm: sampling colony} with F(x) = x). Thus the spine trajectories are drawn from a simple birth and death process with an immortal colony, that is, with the parameters chosen in Fig.\ref{fig: compare times}B , approximately $10$ times faster to compute. Furthermore, in this example, direct trajectories where the population did not survive are thrown away, and in those who survived, a colony is uniformly picked in the population. These two mechanisms largely increase the value of the variance between trajectories, resulting in an effective convergence approximately $10^3$ times slower compared to spinal method. Employing the spinal method in this context thus yields a significant time efficiency improvement, with a gain on the order of $10^4$.

\begin{figure}[h!]
    \centering
    \includegraphics[width=\linewidth]{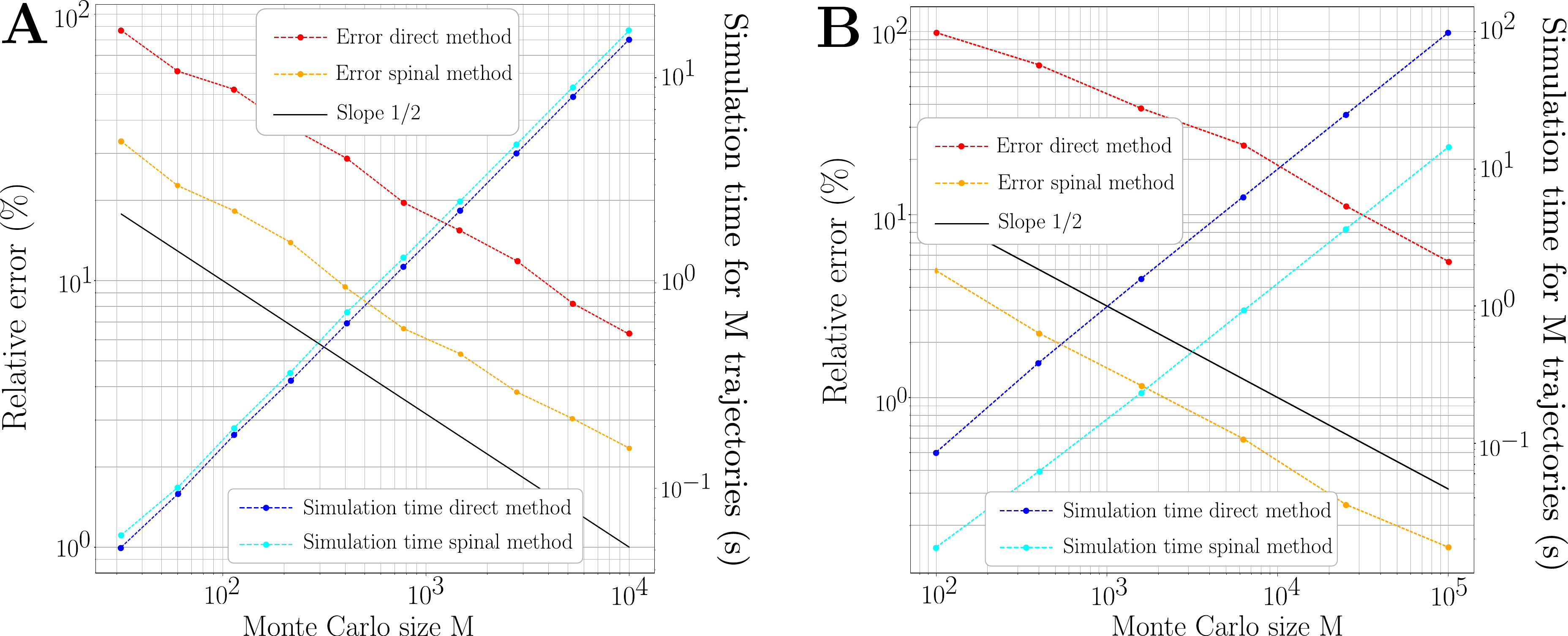}
    \caption{Comparison of simulation times and relative error for Monte Carlo estimates from direct and spinal methods. (A) $\mathbb{V}ar[R_t]$ estimation. Parameters are $a = 1$, $\lambda = 1$, $\mu = 0.3$, $\sigma = 1$, $\delta = 0.5$, $T = 2$. (B) Estimation of $\mathbb{E}[\mathbbm{1}_{\G(t)\neq 0}X^{U_t}]$. Parameters are $a = 1$, $\lambda = 2$, $\mu = 0.5$, $\sigma = 0.5$, $\delta = 0.5$, $T = 10$.}
    \label{fig: compare times}
\end{figure}

\section{General population model and spinal construction}\label{section: general spine}
In the following we will denote by $\mathfrak{B}\left(A,B\right)$ (resp. $\mathcal{C}^i\left(A,B\right)$) the set of measurable (resp. $i$ times continuously differentiable) $B$-valued functions defined on a set $A$. We denote by $\mathcal{U}$ the Ulam-Harris-Neveu set of labels of the individuals, and by $\mathcal{X} \subset \mathbb{R}^d$ the set of traits of the individuals in the population, for $d \geq 1$. In this section, the framework is not restricted to colonies modeling and thus we will use the term individual in all generality.

Let us introduce $\Zbarset$ the set composed of all finite point measures on $\mathcal{U}\times\mathcal{X}$
\begin{equation*}
\Zbarset := \Big\{\sum_{i=1}^{N} \delta_{(u^i,x^i)},\  N \in \mathbb{N}, \ \left(u^i, 1 \leq i \leq N\right) \in  \mathcal{U}^{N}, \  \left(x^i, 1 \leq i \leq N\right) \in  \mathcal{X}^{N}   \Big\}.
\end{equation*}
We also define the set of marginal population measures, that is
\begin{equation*}
\Zset := \Big\{\sum_{i=1}^N \delta_{x^i},\ N \in \mathbb{N}, \ \left(x^i, 1 \leq i \leq N\right) \in  \mathcal{X}^{N} \Big\}.
\end{equation*}
For any measure $\bar{\nu} = \sum_{i=1}^{N} \delta_{(u^i,x^i)}$ in $\Zbarset$, we will write $\nu := \sum_{i=1}^{N} \delta_{x^i}$ its projection on $\Zset$. By convention, if the number of points in the measure is $N=0$, $\bar{\nu}$ and $\nu$ are the trivial zero measures on $\mathcal{U}\times\mathcal{X}$ and $\mathcal{X}$.
We introduce for every $\bar{\nu} \in \Zbarset$, every $g \in \mathfrak{B}\left(\mathcal{U}\times \mathcal{X},\mathbb{R}\right)$ and every $f\in \mathfrak{B}\left(\mathcal{X},\mathbb{R}\right)$  
\begin{equation*}
\langle \bar{\nu}, g \rangle := \int_{\mathcal{U} \times \mathcal{X}}g(u,x)\bar{\nu}(\textrm{d}u,\textrm{d}x), \ \textrm{ and } \ \langle \nu, f \rangle := \int_{\mathcal{X}}f(x)\nu(\textrm{d}x).
\end{equation*}
Finally, we denote by $\mathbb{D}\left(A,B\right)$ the Skorohod space of $B$-valued càdlàg functions on a subset $A$ of $\mathbb{R}_+$. 
For every process $\left(\Z_t, t \in A \right) \in \mathbb{D}\left(A,B\right)$ and $z \in B$, we will denote 
$ \mathbb{E}_z\left[f\left(\Z_t\right) \right] := \mathbb{E}\left[ f\left(\Z_t\right) \vert \Z_0=z\right].$
\subsection{Definition of the population}  \label{section: definition}
The population is described at any time $t \in \mathbb{R}_+$ by the finite point measure $\Zbar_t \in \Zbarset$ giving the label and the trait of every individual living in the population at this time. We introduce 
\begin{equation*}
\mathbb{G}\left( t\right) := \left\{u \in \mathcal{U}: \int_{\mathcal{U}\times\mathcal{X}}\mathbbm{1}_{\{v=u\}}\Zbar_t\left(\textrm{d}v,\textrm{d}x\right) \geq 1 \right\}
\end{equation*} 
the set of labels of living individuals at time $t$. For every individual labeled by $u \in \mathbb{G}\left( t\right)$ we denote by $X^u_t$ its trait, and with a slight abuse of notation, $X^u_s$ will denote the trait of its unique ancestor living at time $s \in [0,t]$.
Using these notations, we can write 
\begin{equation*}
\Zbar_t = \sum_{u \in \mathbb{G}\left( t\right)}\delta_{\left(u,X^u_t\right)} \quad \textrm{and} \quad \Z_t := \sum_{u \in \mathbb{G}\left( t\right)}\delta_{X^u_t}.
\end{equation*}
The initial population is given by a measure $\bar{z} = \sum_{i=1}^{N}\delta_{\left(u^i,x^i\right)}$, and 
the branching rate by a function $B\in\mathfrak{B}(\mathcal{X} \times \Zset,\mathbb{R}_+)$. An individual with trait $x$ in a population $\Z$ dies at an instantaneous rate $B\left(x,\Z\right)$. 
It gives birth to $n$ offspring, where $n$ is randomly chosen with distribution $\left(p_{k}\left(x,\Z\right), k \in \mathbb{N}\right)$ of first moment $m\left(x,\Z\right)$. Hence branching events leading to $n$ children happen at rate $B_n(\cdot) := B(\cdot)p_n(\cdot)$. If $n\geq 0$, we denote by  $\boldsymbol{y}\in\mathcal{X}^n$ the offspring traits at birth, randomly chosen according to the law $K_n(x,\Z,\cdot)$. For all $1 \leq i \leq n$, the $i$-th child of a parent of label $u$ is labeled $ui$ and its trait is $y^i$, the $i$-th coordinate of $\boldsymbol{y}$. 
We denote by $T_{\textrm{Exp}} \in\overline{\mathbb{R}}_+$ the explosion time of the process $\left(\Zbar_t, t\geq 0 \right)$, defined as the limit of its jumps times $(T_k, k\geq 0)$.  

Recall that $N_t$ denotes the number of individuals in the population at time $t\geq 0$. It gives the dimension of the vectorial objects. 
For all $n \geq 1$, we introduce a number $q(n) \geq n$ that denotes the number of sources of noise in the trait dynamics. For all $q \geq n \geq 1$, we introduce $\boldsymbol{a}^{(n)} \in\mathfrak{B}(\Zbarset ,\ \mathcal{X}^n)$ and  $\ssigma^{(n,q)} \in \mathfrak{B}(\Zbarset ,\ M_{n,q}(\mathcal{X}))$ the drift and diffusion coefficients, where $ M_{n,q}(\mathcal{X})$ denotes the set of all $n \times q$ matrices with coefficients in $\mathcal{X}$. 
We also denote by $\boldsymbol{X}_t := (X^u_t, u \in\mathbb{G}(t))$ the vector of $\mathcal{X}^{N_t}$ containing the traits of every living individuals in the population.
Between two successive branching events, the traits of the individuals alive in the population evolve according to population-dependent dynamics
\begin{equation}\label{eq: eds trait}
\text{d}\boldsymbol{X}_t = \boldsymbol{a}^{\left(N_t\right)}\left(\Zbar_t\right)\text{d}t + \ssigma^{\left(N_t, q\left(N_t\right)\right)}\left(\Zbar_t\right) \text{d}\boldsymbol{B}_t
\end{equation}
where $\mathbf{B}_t$ is a $q\left(N_t\right)$-dimensional Wiener process. The case $q(n) > n$ corresponds to an external noise added to the trait dynamics of each individual, that can be for example, the impact of a common environment that correlates the individuals. Such a process can be seen as a generalization of a  multi-type continuous-state branching process in random environment \cite{barczy14, bansaye21}.
The population model introduced in Section \ref{section: example} and studied in this paper is an example of a process satisfying \eqref{eq: eds trait}.

The formalism used in \eqref{eq: eds trait} also allows to consider models where the individual dynamics differ between individuals, and change with the population size. Notice that this formalism  will be useful to derive further expressions of operators and generators of the branching process.

In Section \ref{section: Many-to-One formula}, we establish a $\psi$-spinal construction that distinguishes an individual in the population, and derive a Many-to-One formula in the context of branching processes with complex interactions. To this end we will need additional notations.
The spinal construction generates a $\Zbarset$-valued process along with the label of a distinguished individual that can change with time. For convenience, we will denote by $\Zbarhatset$ and $\Zhatset$ the sets such that
\begin{equation*}\label{Z hat space}
\Zbarhatset := \left\{\left(\mathfrak{e},\Zbar\right) \in \mathcal{U} \times \Zbarset: \ \langle \Zbar, \mathbbm{1}_{\{ \mathfrak{e}\}\times\mathcal{X}} \rangle\geq 1  \right\}, \quad \textrm{and } \quad  
\Zhatset := \left\{\left(x,\Z\right) \in \mathcal{X} \times \Zset: \ \langle \Z, \mathbbm{1}_{\{ x\}} \rangle \geq 1  \right\}.
\end{equation*}
For every $\left(\mathfrak{e},\Zbar\right) \in \Zbarhatset$, we uniquely define $x_{\mathfrak{e}}$ the trait of the individual of label $\mathfrak{e}$ in the population $\Zbar$.
Thus, the spine process is a $\Zbarhatset$-valued branching process and its marginal is a $\Zhatset$-valued branching process. We will propose a generalized spinal construction, where branching rates are biased with weight functions chosen in a set $\mathcal{D}$, defined by
\begin{equation}\label{def: D space}
\mathcal{D} := \left\{F_f \in \mathfrak{B}\left(\Zbarhatset,\mathbb{R}_+^*\right)\ \textit{s.t.} \  \left(f,F\right) \in \mathcal{C}^2\left(\mathcal{X},\mathbb{R}\right)\times \mathcal{C}^2\left(\mathcal{X}\times\mathbb{R},\mathbb{R}_+^*\right) \right\},
\end{equation}
where for every $\left(\mathfrak{e},\Zbar\right) \in \Zbarhatset$,  $F_f\left(\mathfrak{e},\Zbar\right) := F(x_\mathfrak{e},\langle\Z,f\rangle)$. In order to alleviate the notations when there is no ambiguity, we will omit the subscript $f$. 
In the following we will also omit the superscript indicating the dimension of the drift and diffusion functions, and refer to the dynamics of the individual $u$ with a subscript $u$ in those functions.
\begin{proposition}\label{prop: generator G}
The infinitesimal generator $G$ of the trait process solution of \eqref{eq: eds trait}, defined on the set $\mathcal{D}$ is given, for all functions $F_f \in \mathcal{D}$ and all $\left(\mathfrak{e},\Zbar\right) \in \Zbarhatset$ by:
\begin{align}\label{eq: generator}
GF_f\left(\mathfrak{e},\Zbar\right) &= \partial_1F\left(x_{\mathfrak{e}},\langle \nu,f \rangle \right)\boldsymbol{a}_{\mathfrak{e}}\left(\Zbar\right)+\frac{1}{2}\partial^2_{1,1}F\left(x_{\mathfrak{e}},\langle \nu,f \rangle \right)\left(\ssigma\ssigma^T\right)_{\mathfrak{e},\mathfrak{e}}\left(\Zbar\right) \nonumber\\
&+ \partial_{1,2}^2F\left(x_{\mathfrak{e}},\langle \nu,f \rangle \right) \int_{\mathcal{X}\times\mathcal{U}} \left(\ssigma\ssigma^T\right)_{u,\mathfrak{e}}\left(\Zbar\right)f'\left(x\right) \Zbar(\text{d}x,\text{d}u)\nonumber\\
&+ \partial_2F\left(x_{\mathfrak{e}},\langle \nu,f \rangle \right) \int_{\mathcal{X}\times\mathcal{U}}\Big(\boldsymbol{a}_u\left(\Zbar\right)f'\left(x\right) + \frac{\left(\ssigma\ssigma^T\right)_{u,u}}{2}\left(\Zbar\right)f''\left(x\right)\Big)\Zbar(\text{d}x,\text{d}u)\nonumber\\
&+\frac{1}{2}\partial^2_{2,2}F\left(x_{\mathfrak{e}},\langle \nu,f \rangle \right)\int_{\mathcal{X}\times\mathcal{U}}\int_{\mathcal{X}\times\mathcal{U}}\left(\ssigma\ssigma^T\right)_{u,v}\left(\Zbar\right)f'\left(x\right)f'\left(y\right)\Zbar(\text{d}x,\text{d}u)\Zbar(\text{d}y,\text{d}v).
\end{align}    
\end{proposition}
This proposition is proven in Section \ref{section: proof part 1}. The first two lines are related to the dynamic of a distinguished individual in the population. The first order derivative describes the drift of its trait and the second order derivatives are the diffusion parts. Notice that, thanks to the interactions, the diffusion dynamics are separated in an autocorrelated part (in the first line) and a sum of the correlations from the other individuals in the second line. 
The last two lines describe the dynamics of all the individuals- also with the distinguished one- in a similar way.

\begin{corollary}
When the individuals follow the same dynamics and are not correlated through the Brownian motion, the drift vector becomes $\boldsymbol{a}(\Zbar_t) = (a(X^u_t,\Z_t), u \in \mathbb{G}(t))$ and the diffusion matrix is diagonal; $\ssigma(\Zbar_t) = \text{diag}(\sigma(X^u_t,\Z_t), u \in \mathbb{G}(t))$. In this case the infinitesimal generator $G$ of the trait process is given, for all functions $F_f \in \mathcal{D}$ and all $\left(\mathfrak{e},\Zbar\right) \in \Zbarhatset$ by:
\begin{align*}
GF_f\left(\mathfrak{e},\Zbar\right) &= \partial_1F\left(x_{\mathfrak{e}},\langle \nu,f \rangle \right)a\left(x_{\mathfrak{e}},\Z\right)+\frac{1}{2}\partial^2_{1,1}F\left(x_{\mathfrak{e}},\langle \nu,f \rangle \right)\sigma^2\left(x_{\mathfrak{e}},\Z\right) \nonumber\\
&+ \partial_2F\left(x_{\mathfrak{e}},\langle \nu,f \rangle \right) \int_{\mathcal{X}} a\left(x,\Z\right)f'\left(x\right) + \frac{1}{2}f''\left(x\right)\sigma^2\left(x,\Z\right)\Z(\text{d}x).
\end{align*}    
\end{corollary}
This form is useful in many models without interactions or when the trait dynamics are indistinguishable between individuals \cite{BT11, barczy14}. 

\subsection{$\psi$-spine construction} \label{section: Many-to-One formula}
For $(u,\Zbar) \in \Zbarhatset$ we recall that $x_u$ is the trait of the individual of label $u$ in the population $\Zbar$, and for $n\geq 0$ and $\boldsymbol{y}=(y^i,1\leq i \leq n) \in \mathcal{X}^n$ we will denote by
\begin{equation}\label{nu+}
\Zbar_+(u,\boldsymbol{y}) := \Zbar -\delta_{(u,x_u)} + \sum_{i=1}^{n} \delta_{(ui,y^i)}, \quad \textrm{and} \quad \Z^+(x,\boldsymbol{y}) := \Z -\delta_{x} + \sum_{i=1}^{n} \delta_{y^i}
\end{equation}
the state of the population, just after a branching event replacing the individual $u$ by the offspring of traits $\mathbf{y}$.
We now introduce the key operator $\mathcal{G}$ involved in the spinal construction. It is defined for all $F \in \mathcal{D}$ and $(\mathfrak{e},\Zbar) \in \Zbarhatset$ by
\begin{multline}\label{def: mathcal G}
\mathcal{G}F(\mathfrak{e},\Zbar) := GF\left(\mathfrak{e},\Zbar\right)  + \sum_{n\geq 0}  B_n(x_{\mathfrak{e}},\Z)\int_{\mathcal{X}^n} \Big[ \sum_{i=1}^n F \left(\mathfrak{e}i, \Zbar^+(x_{\mathfrak{e}},\boldsymbol{y})\right) 
- F\left(\mathfrak{e}, \Zbar\right)\Big] K_n\left(x_{\mathfrak{e}},\Z,\text{d}\boldsymbol{y}\right)  \\
+ \sum_{n\geq 0} \int_{\mathcal{X}\times\mathcal{U}\backslash\{\mathfrak{e}\}} B_n(x,\Z) \int_{\mathcal{X}^n}\left[ F\left(\mathfrak{e}, \Zbar^+(u,\boldsymbol{y})\right) - F(\mathfrak{e},\Zbar) \right]K_n\left(x,\Z,\text{d}\boldsymbol{y}\right)\Zbar(\text{d}x,\text{d}u),
\end{multline}
where $G$ is the generator of the traits evolution between branching events defined in \eqref{eq: generator}. The first line corresponds to the branching event of the distinguished individuals and the choice of the new one among its children. The second line describes the branching events for every other individuals. By convention we consider that $K_0$ is the null measure on the set $\mathcal{X}^0 = \emptyset$.

Note that the operator $\mathcal{G}$ is generally not the generator of a conservative Markov process on $\Zhatset$. We recall that every function $F_f$ in the set $\mathcal{D}$ defined in \eqref{def: D space} is such that $F$ is $\mathcal{C}^2\left(\mathcal{X}\times\mathbb{R},\mathbb{R}_+^*\right)$, thus the functions $\mathcal{G}F$ are locally bounded on $\Zbarhatset$.

Now we can introduce the $\Zbarhatset$-valued spine process associated with a non negative bias function $\psi_{\phi} \in \mathcal{D}$, that will be called the $\psi_{\phi}$-spine process in order to keep track of the function used in its construction.
For such functions $\psi_{\phi}$, the generator $\widehat{\mathcal{L}}_{\psi_{\phi}}$ of the $\psi_{\phi}$- spine process is defined for all $F_f \in \mathcal{D}$ and all $(\mathfrak{e},\Zbar) \in \Zbarhatset\times\mathbb{R}_+$ by
\begin{equation}\label{generator spinal}
\widehat{\mathcal{L}}_{\psi_{\phi}}F_f(\mathfrak{e},\Zbar) := \frac{\mathcal{G}\left[\psi_{\phi} F_f\right]}{\psi_{\phi}}(\mathfrak{e},\Zbar) - \frac{\mathcal{G}\psi_{\phi}}{\psi_{\phi}}(\mathfrak{e},\Zbar)  F_f(\mathfrak{e},\Zbar).   
\end{equation}
This generator can also be expressed in the following form, which is more descriptive and allows for a better understanding of the dynamics of individuals in the spinal process.  
\begin{proposition}\label{prop: generator}
The generator of the spinal process satisfies for $F_f \in \mathcal{D}$ and $(\mathfrak{e},\Zbar) \in \Zbarhatset\times\mathbb{R}_+$
\begin{multline*}
\widehat{\mathcal{L}}_{\psi_{\phi}}F_f(\mathfrak{e},\Zbar) := \widehat{G}F_f(\mathfrak{e},\Zbar)   \\
 + \sum_{n\geq 0}  B_n(x_{\mathfrak{e}},\Z)\int_{\mathcal{X}^n} \sum_{i=1}^n  \left[F_f \left(\mathfrak{e}i, \Zbar^+(x_{\mathfrak{e}},\boldsymbol{y})\right) 
- F_f\left(\mathfrak{e}, \Zbar\right)\right] \frac{\psi_{\phi}\left(y^i,\Z^+(x_{\mathfrak{e}},\boldsymbol{y})\right)}{\psi_{\phi}(x_{\mathfrak{e}},\Z)}K_n\left(x_{\mathfrak{e}},\Z,\text{d}\boldsymbol{y}\right) \\
+ \sum_{n\geq 0} \int_{\mathcal{X}\times\mathcal{U}\backslash\{\mathfrak{e}\}} B_n(x,\Z) \int_{\mathcal{X}^n} \left[F_f\left(u, \Zbar^+(\mathfrak{e},\boldsymbol{y})\right)\frac{\psi_{\phi}\left(x_{\mathfrak{e}},\Z^+(x_u,\boldsymbol{y})\right)}{\psi_{\phi}(x_{\mathfrak{e}},\Z)}- F_f(\mathfrak{e},\Zbar) \right] \hspace{1.45cm} \\
\times K_n\left(x,\Z,\text{d}\boldsymbol{y}\right) \Zbar(\text{d}x,\text{d}u)
\end{multline*}
where $\nu^+$ is defined in \eqref{nu+}. The continuous part is given by (with $\mathfrak{X}=\mathcal{X}\times\mathcal{U}$)
\begin{multline} \label{widehatG}
 \widehat{G}F_f(\mathfrak{e},\Zbar) = \partial_1F\left( \boldsymbol{a}_{\mathfrak{e}} + \frac{\partial_1\psi}{\psi_{\phi}} \left(\ssigma\ssigma^T\right)_{\mathfrak{e},\mathfrak{e}} + \frac{\partial_2\psi}{\psi_{\phi}}  \int_{\mathfrak{X}}\left(\ssigma\ssigma^T\right)_{\mathfrak{e},v}\phi'(x) \Zbar(\text{d}x,\text{d}u)\right) \\ \nonumber
+\frac{1}{2}\partial^2_{1,1}F\left(\ssigma\ssigma^T\right)_{\mathfrak{e},\mathfrak{e}} +  \partial_{1,2}^2F \int_{\mathfrak{X}} \left(\ssigma\ssigma^T\right)_{u,\mathfrak{e}}f'\left(x\right) \Zbar(\text{d}x,\text{d}u)\nonumber\\
+ \partial_2F \int_{\mathfrak{X}}f'\left(x\right)\left[\boldsymbol{a}_u + \frac{\partial_1\psi}{\psi_{\phi}}  \left(\ssigma\ssigma^T\right)_{u,\mathfrak{e}} + \frac{\partial_2\psi}{\psi_{\phi}} \int_{\mathfrak{X}}\left(\ssigma\ssigma^T\right)_{u,v}\phi'(y) \Zbar(\text{d}y,\text{d}v) \right]\Zbar(\text{d}x,\text{d}u)  \\
+\partial_2F \int_{\mathfrak{X}}\frac{\left(\ssigma\ssigma^T\right)_{u,u}}{2}f''\left(x\right)\Zbar(\text{d}x,\text{d}u)
+\frac{1}{2}\partial^2_{2,2}F\int_{\mathfrak{X}}\int_{\mathfrak{X}}\left(\ssigma\ssigma^T\right)_{u,v}f'\left(x\right)f'\left(y\right)\Zbar(\text{d}x,\text{d}u)\Zbar(\text{d}y,\text{d}v). \nonumber
\end{multline} 
We omitted the dependencies $\left(x_{\mathfrak{e}},\langle \nu,f \rangle \right)$ and $(\Zbar)$ to shorten the notations, and used the notations
$$
\partial_i\psi := \partial_i\psi\left(x_{\mathfrak{e}},\langle \nu,\phi \rangle \right), \quad \text{for } i \in \{1, 2\}.
$$
\end{proposition}
From the expression of the continuous part of the generator, we get that the traits dynamics between jumps in the spinal process are the same than those of the original process with an added drift in the direction $\ssigma\ssigma^T\nabla\psi_{\phi}$:
\begin{equation}\label{eq: added drift}
    \widehat{G}F_f(\mathfrak{e},\Zbar) = GF_f(\mathfrak{e},\Zbar) + \Big(\frac{ \ssigma\ssigma^T\nabla\psi_{\phi}}{\psi_{\phi}}\Big)^T \nabla F_f(\mathfrak{e},\Zbar)
\end{equation}
where the gradient is defined as
$$
\nabla F_f(\mathfrak{e},\Zbar) := \partial_1 F\left(x_{\mathfrak{e}},\langle \nu,f \rangle \right) + \partial_2 F\left(x_{\mathfrak{e}},\langle \nu,f \rangle \right)\int_{\mathcal{X}\times\mathcal{U}}f'\left(x\right)\Zbar(\text{d}x,\text{d}u).
$$
The proof of Proposition \ref{prop: generator} and the vectorial expression of the added drift are displayed in Section \ref{section: proof part 1}.
Notice that the branching rates of both spinal and non-spinal individuals are biased by the function $\psi_{\phi}$. Dynamics 2.3 and 2.4 in \cite{medous23} provide a more detailed explanation of the branching mechanisms in the spine process associated with this generator.
The spinal construction is based on a martingale change of measure, that necessitates a moment assumption on the original process and the function $\psi_{\phi}$ involved in this change of measure. A simple condition to ensure this moment assumption is a boundary condition on the added drift. 
\begin{assumption}\label{ass: novikov}
There exists $M\in\mathbb{R}$ such that for all $(e,\Zbar) \in \Zbarhatset\times\mathbb{R}_+$
\begin{multline*}
        \left(\partial_2\psi\right)^2 (e,\Zbar)\int_{\mathcal{X}\times\mathcal{U}}\int_{\mathcal{X}\times\mathcal{U}}\left(\ssigma\ssigma^T\right)_{u,v}\phi'(y)\phi'(x) \Zbar(\text{d}y,\text{d}v) \Zbar(\text{d}x,\text{d}u)\\
    + \left(\partial_1\psi\right)^2(e,\Zbar)\left(\ssigma\ssigma^T\right)_{\e,\e} + 2 \partial_1\psi\partial_2\psi (e,\Zbar)\int_{\mathcal{X}\times\mathcal{U}}\phi'(x)\left(\ssigma\ssigma^T\right)_{u,\e} \Zbar(\text{d}x,\text{d}u) \leq M \psi(e,\Zbar).
\end{multline*}
\end{assumption}
This general assumption does not take into account the branching dynamics of the spinal process. It can be lightened by considering an assumption on a path-integral exponential moment of the spinal process, under the form of the Novikov condition. For more details, we refer the interested reader to the proof of Theorem \ref{thm: sampling} in Section \ref{proof pdmc}.
Finally, we introduce for all $t\geq 0$, the $\mathcal{U}$-valued random variable $U_t$ that picks an individual alive at time $t$. Its law is characterized by the function $p_u\left(\Zbar_t\right)$ which yields the probability to choose the individual of label $u$ in the set $\mathbb{G}\left(t\right)$.
We can now state the main result of this section, that is a Girsanov-type formula for the spinal change of measure. It characterizes the joint probability distribution of $\left(U_t,(\Zbar_s,s\leq t)\right)$- that is the randomly sampled individual in the population $\Zbar_t$ at time $t$ and the whole trajectory of the population until this time- and links it to the law of the spine process through a path-integral formula. 
\begin{theorem}\label{thm: sampling}
Let $\psi_{\phi} \in\mathcal{D}$ satisfying Assumption \ref{ass: novikov}, $t \geq 0$, $\bar{z} \in \Zbarset$. 
Let $((E_t,\Zbarhat_{t}), t\geq 0)$ be the $\Zbarhatset$-valued branching process with interactions defined by the infinitesimal generator $\widehat{\mathcal{L}}_{\psi_{\phi}}$ introduced in \eqref{generator spinal}, and $\widehat{T}_{\textrm{Exp}}$ its explosion time.
For all measurable function $H$ on $\mathcal{U} \times \mathbb{D}\left([0,t],\Zset\right)$:
\begin{multline*}
\mathbb{E}_{\bar{z}}\left[\mathbbm{1}_{\{T_{\textrm{Exp}} > t, \mathbb{G}(t) \neq \emptyset\}}H\left(U_t,(\Zbar_s,s\leq t)\right) \right] = \\ \langle z,\psi_{\phi}(\cdot,z) \rangle \mathbb{E}_{\bar{z}}\left[\mathbbm{1}_{\{\widehat{T}_{\textrm{Exp}}>t\}} \xi\left(E_t,\left(\Zbarhat_s, s\leq t\right)\right) H\left(E_t,\left(\Zbarhat_s, s\leq t\right)\right)\right],
\end{multline*}
where:
\begin{equation*}
\xi\left(E_t,\left(\Zbarhat_s, s\leq t\right)\right) :=  \frac{p_{E_t}(\Zbarhat_t)}{\psi_{\phi}(E_t,\Zbarhat_t)}\exp\Bigg(\int_0^t \frac{\mathcal{G} \psi_{\phi}\left(E_s,\Zbarhat_s\right)}{\psi_{\phi}(E_s,\Zbarhat_s)}\textrm{d}s\Bigg).
\end{equation*}
\end{theorem}
The proof of this result is an extension of the proof for the case $\ssigma = 0$ in \cite{medous23}. We show by induction on the successive jumps times the equality in law for the trajectories between these times, using a Girsanov theorem to take into account the Brownian trajectories. 
 
The process $\left(\left(E_t,\Zbarhat_{t}\right), t\geq 0 \right)$ gives at any time $t\geq0$ the label of the spinal individual- that encodes the whole spine lineage- and the spinal population. Our result thus links, for every $\psi_{\phi}$, the sampling of an individual and the trajectory of the population to the trajectory of the spine process. This formula gives a change of probability that involves the function $\mathcal{G}\psi_{\phi}/\psi_{\phi}$ with a path-integral formula, related to Feynman-Kac path measures and semi-groups \cite{DelMoral04}.  The path integral term that links these two terms is difficult to handle in general and finding eigenfunctions of $\mathcal{G}$ allows a dual representation of the first moment semi-group of the original population with the generator of the spinal individual \cite{Athreya00, englander2010strong, Cloez17, B21}.
\begin{corollary}\label{cor: eigenfunction mto}
Let $\psi_{\phi} \in\mathcal{D}$ satisfying Assumption \ref{ass: novikov}, such that there exists $\lambda_0 \in \mathbb{R}$ satisfying for all $(e,\bar{z}) \in \Zbarhatset$. 
\begin{equation}\label{condition eigenfunction}
    \frac{\mathcal{G} \psi_{\phi}\left(e,\bar{z}\right)}{\psi_{\phi}\left(e,\bar{z}\right)} = \lambda_0.
\end{equation}
Let $\left(\left(E_t,\Zbarhat_{t}\right), t\geq 0 \right)$ be the $\Zbarhatset$-valued branching process with interactions defined by the infinitesimal generator $\widehat{\mathcal{L}}_{\psi_{\phi}}$. Let $t \geq 0$, and $\bar{z} \in \Zbarset$.
If the original process and the $\psi_{\phi}$-spine process do not explose a.s. before time $t$, then for every measurable function $f$ on $\mathcal{U} \times \mathbb{D}\left([0,t],\mathcal{X}\right)$:
\begin{equation*}
\mathbb{E}_{\bar{z}}\Bigg[\sum_{u \in \G(t)}f\left(X^u_s,s\leq t\right) \Bigg] =  \langle z,\psi_{\phi}(\cdot,\bar{z}) \rangle e^{\lambda_0 t}\mathbb{E}_{\bar{z}}\Bigg[ \frac{1}{\psi_{\phi}(E_t,\Zbarhat_t)} f(\widehat{Y}_t, s\leq t)\Bigg].
\end{equation*}
\end{corollary}
\begin{proof}
Let $f$ be a measurable function on $\mathcal{U} \times \mathbb{D}\left([0,t],\mathcal{X}\right)$. We introduce $H$ the measurable function defined for all $(u,\bar{z}_s, s\leq t) \in\mathcal{U}\times \mathbb{D}\left([0,t],\Zbarset\right)$ by
$$
H(u,\bar{z}_s, s\leq t) := f\left(X^u_s, s\leq t\right)\langle z_t,1\rangle.
$$
The corollary is thus a direct application of Theorem \ref{thm: sampling} with the function $H$ and a uniformly sampled individual under non-explosion conditions, see \textit{i.e.} \cite{marguet19} for sufficient conditions. 
\end{proof}
For structured branching processes without interactions in some particular cases, it is possible to exhibit functions satisfying \eqref{condition eigenfunction}. For instance, for Ornstein-Uhlenbeck processes, Englander \cite{englander2010strong} derived eigenfunctions in the case of constant or quadratic breeding rate, and Cloez \cite{Cloez17} in the affine breeding case. In more general cases, it is possible to partially solve the eigenproblem and derive only the eigenvalue $\lambda_0$ that gives the exponential growth of the population \cite{Cloez17}[Section 6]. In the case of non-structured branching processes with size-dependent reproduction rates, \textit{i.e.} for all $k\geq 0$, $B_k(x,\Z) = B_k(\langle \Z, 1 \rangle)$, Bansaye \cite{B21}[Section 3] established that the function $\psi_{\phi}(x,\Z) = 1/\langle \Z, 1 \rangle$ always satisfies \eqref{condition eigenfunction}. However, finding such functions in the structured case with interactions is particularly complex as the operator $\mathcal{G}$ is non-linear.

\section{Proofs}
The last section is dedicated to the proofs of the results collected in the different sections of the paper. We begin by presenting the proofs from Section \ref{section: general spine}, which are essential for later establishing the results discussed in Section \ref{section: example}.
\subsection{Proofs of Section \ref{section: general spine}}\label{section: proof part 1}
\subsubsection{Proof of Proposition \ref{prop: generator G}}
We denote by $\Zbar_t$ the population at time $t$ and we choose $\mathfrak{e}\in \mathbb{G}(t)$ for the label of the distinguished individual. Using Itô formula we get that  
\begin{align}\label{eq: dF_f}
\text{d}F_f\left(\mathfrak{e},\Zbar_t\right) &= \text{d}F\Big(X^{\mathfrak{e}}_t,\sum_{v \in \mathbb{G}(t)}f\left(X^v_t\right)\Big) \nonumber\\
& = \sum_{u \in \mathbb{G}(t)} \partial_{u}F\left(X^{\mathfrak{e}}_t,\langle\Z_t,f \rangle \right)a_u\left(\Zbar_t \right) \text{d}t + \sum_{j=1}^m\sum_{u \in \mathbb{G}(t)} \partial_{u}F\left(X^{\mathfrak{e}}_t,\langle\Z_t,f \rangle \right)\ssigma_{u,j}\left(\Zbar_t \right) \text{d}B^j_t \nonumber\\
& + \frac{1}{2} \sum_{u,w \in \mathbb{G}(t)} \partial^2_{u,w}F\left(X^{\mathfrak{e}}_t,\langle\Z_t,f \rangle \right)\left(\ssigma\ssigma^T\right)_{u,w}\left(\Zbar_t \right) \text{d}t.
\end{align}
To derive the generator we will need to compute the following terms. For all $u \neq \mathfrak{e}$
\begin{align}\label{eq: gradient F_f}
\partial_{u}F\left(X^{\mathfrak{e}}_t,\langle\Z_t,f \rangle \right)&= \partial_2F\left(X^{\mathfrak{e}}_t,\langle\Z_t,f \rangle \right)f'\left(X^u_t\right) \nonumber\\
\partial_{\mathfrak{e}}F\left(X^{\mathfrak{e}}_t,\langle\Z_t,f \rangle \right) &= \partial_1F\left(X^{\mathfrak{e}}_t,\langle\Z_t,f\rangle \right) + \partial_2F\left(X^{\mathfrak{e}}_t,\langle\Z_t,f \rangle \right)f'\left(X^{\mathfrak{e}}_t\right).
\end{align}
We deduce that
\begin{multline*}
    \sum_{u \in \mathbb{G}(t)} \partial_{u}F\left(X^{\mathfrak{e}}_t,\langle\Z_t,f \rangle \right)a_u\left(\Zbar_t \right) = \partial_1F\left(X^{\mathfrak{e}}_t,\langle\Z_t,f\rangle \right)a_{\mathfrak{e}}\left(\Zbar_t \right)  + \partial_2F\left(X^{\mathfrak{e}}_t,\langle\Z_t,f \rangle \right)\sum_{u \in \mathbb{G}(t)} f'\left(X^u_t\right)a_u\left(\Zbar_t \right).
\end{multline*}
Let us denote by $\mathring{\mathbb{G}}(t)$ the set of labels without the spine $ \mathbb{G}(t)\backslash\{E_t\}$. Then for $u \neq v \in \mathring{\mathbb{G}}(t)$, and when $E_t=\mathfrak{e}$, the second order derivatives write
\begin{align*}\label{eq: hess F_f}
\partial^2_{u,u}F\left(X^{\mathfrak{e}}_t,\langle\Z_t,f \rangle \right) &= \partial_2F\left(X^{\mathfrak{e}}_t,\langle\Z_t,f \rangle \right)f''\left(X^u_t\right) + \partial^2_{2,2}F\left(X^{\mathfrak{e}}_t,\langle\Z_t,f \rangle \right)\left(f'\left(X^u_t\right)\right)^2\nonumber\\
\partial^2_{u,v}F\left(X^{\mathfrak{e}}_t,\langle\Z_t,f \rangle \right) &= \partial^2_{2,2}F\left(X^{\mathfrak{e}}_t,\langle\Z_t,f \rangle \right)f'\left(X^u_t\right)f'\left(X^v_t\right) \nonumber\\
\partial^2_{u,e}F\left(X^{\mathfrak{e}}_t,\langle\Z_t,f \rangle \right) &= \partial^2_{2,2}F\left(X^{\mathfrak{e}}_t,\langle\Z_t,f \rangle \right)f'\left(X^u_t\right)f'\left(X^{\mathfrak{e}}_t\right) + \partial^2_{1,2}F\left(X^{\mathfrak{e}}_t,\langle\Z_t,f \rangle \right)f'\left(X^u_t\right)\nonumber\\
\partial^2_{e,e}F\left(X^{\mathfrak{e}}_t,\langle\Z_t,f \rangle \right) &= \partial_{2}F\left(X^{\mathfrak{e}}_t,\langle\Z_t,f \rangle \right)f''\left(X^{\mathfrak{e}}_t\right) + \partial^2_{2,2}F\left(X^{\mathfrak{e}}_t,\langle\Z_t,f \rangle \right)\left(f'\left(X^{\mathfrak{e}}_t\right)\right)^2 \nonumber\\
& + \partial^2_{1,1}F\left(X^{\mathfrak{e}}_t,\langle\Z_t,f \rangle \right) + 2\partial^2_{1,2}F\left(X^{\mathfrak{e}}_t,\langle\Z_t,f \rangle \right)f'\left(X^{\mathfrak{e}}_t\right).
\end{align*}
We will denote by $M$ the symmetrical matrix $\ssigma\ssigma^T$ and we omit the dependency in the functions when there is no ambiguity. We thus have 
\begin{align*}
    \frac{1}{2}\sum_{u,w \in \mathbb{G}(t)} \partial^2_{u,w}F M_{u,w} &= 
\frac{1}{2}\sum_{u,w \in \mathring{\mathbb{G}}(t)} \partial^2_{u,w}FM_{u,w} 
+ \sum_{u \in \mathring{\mathbb{G}}(t)} \partial^2_{u,\mathfrak{e}}FM_{u,\mathfrak{e}} + \frac{1}{2}\partial^2_{\mathfrak{e},\mathfrak{e}}FM_{\mathfrak{e},\mathfrak{e}}\\
&= \frac{1}{2}\partial^2_{2,2}F\sum_{u,w \in \mathring{\mathbb{G}}(t)} f'\left(X^u_t\right)f'\left(X^v_t\right)M_{u,w} + \frac{1}{2}\partial_{2}F\sum_{u \in \mathring{\mathbb{G}}(t)} f''\left(X^u_t\right)M_{u,\mathfrak{e}}\\
& \quad + \partial^2_{1,2}F\sum_{u \in \mathring{\mathbb{G}}(t)} f'\left(X^u_t\right)M_{u,\mathfrak{e}} + \partial^2_{2,2}F\sum_{u \in \mathring{\mathbb{G}}(t)} f'\left(X^u_t\right)f'\left(X^{\mathfrak{e}}_t\right)M_{u,\mathfrak{e}}\\
& \quad + \frac{1}{2}\partial_{2}Ff''\left(X^{\mathfrak{e}}_t\right)M_{\mathfrak{e},\mathfrak{e}} + \frac{1}{2} \partial^2_{2,2}f'\left(X^{\mathfrak{e}}_t\right)f'\left(X^{\mathfrak{e}}_t\right)M_{\mathfrak{e},\mathfrak{e}}.
\end{align*}
Using the definition of the infinitesimal generator and rearranging the terms conclude the proof.

\subsubsection{Proof of Proposition \ref{prop: generator}}
The equality for the jumps part of the generator is Proposition 2.6 in \cite{medous23}. For the continuous part $G$, we use \eqref{eq: dF_f} to get that 
\begin{align*}
G[\psi F]\left(\mathfrak{e},\Zbar\right) &= \psi\left(X^{\mathfrak{e}}_t,\langle\Z_t,f \rangle \right)\sum_{u \in \mathbb{G}(t)} \partial_{u}F\left(X^{\mathfrak{e}}_t,\langle\Z_t,f \rangle \right)a_u\left(\Zbar_t \right) \\
&+ F\left(X^{\mathfrak{e}}_t,\langle\Z_t,f \rangle \right)\sum_{u \in \mathbb{G}(t)} \partial_{u}\psi\left(X^{\mathfrak{e}}_t,\langle\Z_t,f \rangle \right)a_u\left(\Zbar_t \right)\\
& + \frac{\psi\left(X^{\mathfrak{e}}_t,\langle\Z_t,f \rangle \right)}{2} \sum_{u,w \in \mathbb{G}(t)} \partial^2_{u,w}F\left(X^{\mathfrak{e}}_t,\langle\Z_t,f \rangle \right)\left(\ssigma\ssigma^T\right)_{u,w}\left(\Zbar_t \right)\\
& + \frac{F\left(X^{\mathfrak{e}}_t,\langle\Z_t,f \rangle \right)}{2} \sum_{u,w \in \mathbb{G}(t)} \partial^2_{u,w}\psi\left(X^{\mathfrak{e}}_t,\langle\Z_t,f \rangle \right)\left(\ssigma\ssigma^T\right)_{u,w}\left(\Zbar_t \right)\\
& + \frac{1}{2} \sum_{u,w \in \mathbb{G}(t)} \left(\partial_{w}F \partial_u\psi + \partial_uF\partial_w\psi \right)\left(X^{\mathfrak{e}}_t,\langle\Z_t,f \rangle \right) \left(\ssigma\ssigma^T\right)_{u,w}\left(\Zbar_t \right).
\end{align*}
Using the symmetry of $\ssigma\ssigma^T$ we get that
\begin{multline*}
    \frac{G[\psi F]}{\psi}\left(\mathfrak{e},\Zbar\right) - \frac{G\psi }{\psi}F\left(\mathfrak{e},\Zbar\right) = GF\left(\mathfrak{e},\Zbar\right) \\ + \frac{1}{\psi\left(X^{\mathfrak{e}}_t,\langle\Z_t,f \rangle \right)}\sum_{u,w \in \mathbb{G}(t)} \partial_u\psi\left(X^{\mathfrak{e}}_t,\langle\Z_t,f \rangle \right) \left(\ssigma\ssigma^T\right)_{u,w}\left(\Zbar_t \right)\partial_{w}F\left(X^{\mathfrak{e}}_t,\langle\Z_t,f \rangle \right).
\end{multline*}
Notice that the last line can also be expressed in matricial form as
$$
\Big(\frac{ \ssigma\ssigma^T\nabla\psi}{\psi}\Big)^T \nabla F.
$$
In the following we omit the dependencies in the functions to shorten the expressions. We slice the sum by distinguishing the label $\e$, and use the first order derivative expressions \eqref{eq: gradient F_f} to get
\begin{align*}
    \left(\ssigma\ssigma^T\nabla\psi \right)^T \nabla F &= \partial_2F\partial_2\psi \int_{\mathcal{X}\times\mathcal{U}}\int_{\mathcal{X}\times\mathcal{U}}f'(x)\left(\ssigma\ssigma^T\right)_{u,v}\phi'(y) \Zbar(\text{d}y,\text{d}v) \Zbar(\text{d}x,\text{d}u)\\
    &+ \partial_1F\partial_1\psi\left(\ssigma\ssigma^T\right)_{\e,\e} + \partial_2F\partial_1\psi \int_{\mathcal{X}\times\mathcal{U}}f'(x)\left(\ssigma\ssigma^T\right)_{u,\e} \Zbar(\text{d}x,\text{d}u) \\
    &+ \partial_1F\partial_2\psi \int_{\mathcal{X}\times\mathcal{U}}\left(\ssigma\ssigma^T\right)_{\e,u}\phi'(x) \Zbar(\text{d}x,\text{d}u).
\end{align*}
We end the proof by summing this last equation and expression \eqref{eq: generator} of the generator $G$.
\subsubsection{Proof of Theorem \ref{thm: sampling}}\label{proof pdmc}
We obtain Theorem \ref{thm: sampling} with a modification from the proof of Theorem 2.2 in \cite{medous23}, to take into account the diffusive trait dynamics between jumps.

We denote $(U_k, k\geq 0)$ the sequence of $\mathcal{U}$-valued random variables giving the label of the branching individuals at the jumps times $(T_k, k \geq 0)$. Let $(N_k, k\geq 0)$ be the sequence of $\mathbb{N}$-valued random variables giving the number of children at each branching event and we denote  for brevity $A_k := (U_k,N_k)$ for all $k\geq 0$. At the $k$-th branching time $T_k$, we denote $\mathcal{Y}_k$ the $\mathcal{X}^{N_k}$-valued random variable giving the vector of offspring traits. Finally we introduce for all $k \geq 0$, $\mathcal{V}_k:=\left(T_k,A_k,\mathcal{Y}_k\right)$. We similarly define  $((\widehat{T}_k,\widehat{U}_k,\widehat{N}_k,\widehat{\mathcal{Y}}_k), k \geq 0)$, the sequence of jumps times, labels of the branching individual, number of children and trait of these children at birth in the spinal construction. Notice that the distribution of number of children and traits depend on whether the branching individual is the spinal one or not. We will also use for all $k \geq 0$, $\widehat{\mathcal{V}}_k:=(\widehat{T}_k,\widehat{A}_k,\widehat{\mathcal{Y}}_k)$ where $\widehat{A}_k := (\widehat{U}_k,\widehat{N}_k)$. At time $s \in [\widehat{T}_{k-1},\widehat{T}_k)$, the label of the spinal individual is denoted $E_k$ and its trait $\widehat{Y}_k$. For a given initial population $\bar{z}= \sum_{i=1}^{n}\delta_{(i,x^i)} \in \Zbarset$, we use by convention $U_0=\emptyset, \ N_0=n, \ \mathcal{Y}_0=(x^i, 1 \leq i \leq n)$ almost surely. The same convention holds for the spine process. For all $0 \leq k $, we introduce the associated filtrations 
$$
\mathcal{F}_k = \sigma\left(\mathcal{V}_i, 0\leq i\leq k  \right),\ \textrm{and } \ 
\widehat{\mathcal{F}}_k = \sigma\left(E_k, \left(\widehat{\mathcal{V}}_i, 0\leq i\leq k\right)  \right).
$$
The core of the proof of Theorem 2.2 in \cite{medous23}, lies in the following Girsanov-type result 
\begin{lemma}[Lemma 1.5.1 in \cite{medous23}]\label{lemma: medous23}
$\left. \right.$\\
We denote by $\widehat{\Z}$ the spine process.
For any $k >0$, $\bar{z} \in \Zbarset$ and all $a = ((u_i,n_i), 0 \leq i \leq k)$, let $F$ be a measurable function on $\Pi_{i=1}^k\left(\mathbb{R}_+ \times \mathcal{U}\times \mathbb{N} \times \mathcal{X}^{n_i}\right)$. For any $\mathfrak{e} \in \widehat{\mathbb{G}}(\widehat{T}_k)$,
\begin{multline*}
\mathbb{E}_{\bar{z}}  \Big[F\left(\mathcal{V}_i, 0 \leq i \leq k\right)\prod_{i=0}^{k}\mathbbm{1}_{\{A_i = a_i\}}\Big]= \left\langle z, \psi\left(\cdot,z,0\right) \right\rangle\\
\times \mathbb{E}_{\bar{z}} \Big[
\xi_k\left(E_k,\left(\widehat{\mathcal{V}}_i, 0 \leq i \leq k\right) \right) 
F\left(\widehat{\mathcal{V}}_i, 0 \leq i\leq k\right) \mathbbm{1}_{\{E_{k} = e\}} \prod_{i=0}^{k}\mathbbm{1}_{\{\widehat{A}_i = a_i\}} \Big],
\end{multline*}
\begin{equation*}
\text{where} \quad \xi_k\left(E_k,\left(\widehat{\mathcal{V}}_i, 0 \leq i \leq k\right) \right) := \frac{1}{\psi\left(\widehat{Y}_{\widehat{T}_k},\widehat{\Z}_{\widehat{T}_k},\widehat{T}_k\right)}\exp\left(\int_{0}^{\widehat{T}_{k}}\frac{\mathcal{G}\psi\left(\widehat{Y}_s,\widehat{\Z}_s,s\right)}{\psi\left(\widehat{Y}_s,\widehat{\Z}_s,s\right)}\textrm{d}s\right).
\end{equation*}
\end{lemma}
We will prove this Lemma in our extended framework- with diffusive trait dynamics- and the proof of Theorem \ref{thm: sampling} follows the same proof as in \cite{medous23}.
Lemma \ref{lemma: medous23} will be proved by induction on the number of branching events $k$ as in \cite{medous23}, and the core Equation (1.5.10) is shown using Lemma 1.5.2 in \cite{medous23} related to the fundamental theorem of calculus.
However Lemma 1.5.2 in \cite{medous23} does not hold anymore in this framework because of the stochastic dynamics of the traits between branching events. We derive an equivalent result that allows us to conclude the proof of Lemma \ref{lemma: medous23} in the same way. We recall that the set $\mathcal{D}$, defined in \eqref{def: D space}, is composed of twice continuously differentiable and continuous functions $F_f$ such that for all $\left(\mathfrak{e},\Zbar\right) \in \Zbarhatset$,  $F_f\left(\mathfrak{e},\Zbar\right) := F(x_\mathfrak{e},\langle\Z,f\rangle)$.
Without loss of generality on the set $\mathcal{D}$, to simplify the notations and for more coherence between our notations and those in \cite{medous23}, we assume that $\psi_{\phi}$ is a function of the trait of the spine and the marginal population and will denote without distinction 
$$\psi_{\phi}\left(\mathfrak{e},\overline{\widehat{\Z}}_t\right) = \psi_{\phi}\left(\widehat{Y}_t,\widehat{\Z}_t\right).$$
\begin{lemma}
If Assumption \ref{ass: novikov} holds, between two successive jumps, for all $t \in [\widehat{T}_{k-1},\widehat{T}_{k})$, for every measurable function H on the set of trajectories on $[\widehat{T}_{k-1},\widehat{T}_{k})$, we get that
\begin{multline*}
    \mathbb{E}\Big[H(\nu_s, \widehat{T}_{k-1} \leq s\leq t )\Big\vert \widehat{\mathcal{F}}_{k}, \{\nu_{ \widehat{T}_{k-1}} = \widehat{\nu}_{\widehat{T}_{k-1}}\}\Big] = \\
    \mathbb{E}\Bigg[\frac{\psi_{\phi}(\widehat{Y}_{\widehat{T}_{k-1}},\widehat{\nu}_{\widehat{T}_{k-1}})}{\psi_{\phi}(\widehat{Y}_t,\widehat{\nu}_t)}\exp\left\{\int_{\widehat{T}_{k-1}}^t\frac{G\psi_{\phi}}{\psi_{\phi}}(\widehat{Y}_s,\widehat{\nu}_s)\text{d}s\right\}H\left(\widehat{\nu}_s,\widehat{T}_{k-1} \leq s\leq t \right)\Big\vert \widehat{\mathcal{F}}_{k} \Bigg].
\end{multline*}
\end{lemma}
Notice that we did not consider time inhomogeneities in this framework as it would have complicated the notations, and our motivating example does not depend on time. Extending the results to such a non-homogeneous context only necessitates adding a time derivative in the expression of $G$ and in the generator of the spine process, and the rest of the proof still holds.
\begin{proof}
Let $k\geq 1$ and $t \in [\widehat{T}_{k-1},\widehat{T}_{k})$.
Using the expression of the generator $\widehat{G}$ of the trait dynamics between branching events in the spine process, introduced in \eqref{eq: added drift}, we got that the vector $\widehat{\boldsymbol{X}}_t = (\widehat{X}^u_t, u \in \widehat{\G}(t))$ containing the vector of the individuals in the spine process satisfies, for a spine of label $\mathfrak{e}$
\begin{equation}\label{eq: sde spinal}
\text{d} \widehat{\boldsymbol{X}}_t = \Big(\boldsymbol{a}\Big(\overline{\widehat{\Z}}_t\Big) + \frac{\ssigma\ssigma^T\nabla\psi_{\phi}}{\psi_{\phi}}\Big(\mathfrak{e},\overline{\widehat{\Z}}_t\Big) \Big)\text{d} t + \ssigma \left(\overline{\widehat{\Z}}_t\right) \text{d}\boldsymbol{\widehat{B}}_t,    
\end{equation}
where the vector $\boldsymbol{\widehat{B}}_t := (\widehat{B}^u_t, u \in \mathbb{G}(\widehat{\nu}_t) )$ is a Wiener process with i.i.d components under the law $\widehat{\mathbb{P}}$ of the auxiliary process.
We introduce $\left(\ga_s, s\geq 0\right)$ the $\widehat{\mathcal{F}}_t$-adapted process via:
$$
\ga_s := -\frac{\ssigma^T\nabla\psi_{\phi}}{\psi_{\phi}}\left(\mathfrak{e},\overline{\widehat{\Z}}_t\right).
$$
The SDE \eqref{eq: sde spinal} can be written:
\begin{equation}\label{eq: sde spinal 1}
\text{d}\widehat{\boldsymbol{X}}_t = \boldsymbol{a}\left(\overline{\widehat{\Z}}_t\right) \text{d}t + \ssigma \left(\overline{\widehat{\Z}}_t\right) \text{d}\Big(\boldsymbol{\widehat{B}}_t - \int_{\widehat{T}_{k-1}}^t\ga_s\text{d}s\Big).
\end{equation}
Let $\mathbb{P}_{k-1}$ and $\widehat{\mathbb{P}}_{k-1}$ denote respectively the laws of the original and the spinal branching processes conditioned on the filtration $\widehat{\mathcal{F}}_{k-1}$. 
The Girsanov-Cameron-Martin theorem (see \cite[p.72]{Yor09}), applied to \eqref{eq: sde spinal 1} ensures that if the local $\widehat{\mathcal{F}}_t$-martingale $\widehat{L}$ defined by 
\begin{equation}\label{eq: L_t}
    \widehat{L}_t := \exp\Big\{\int_{\widehat{T}_{k-1}}^t\ga_s\cdot\text{d}\boldsymbol{\widehat{B}}_s -\frac{1}{2}\int_{\widehat{T}_{k-1}}^t\Vert  \ga_s \Vert_2^2\text{d}s\Big\}
\end{equation}
is a martingale, thus $(\boldsymbol{\widehat{B}}_t - \int_{\widehat{T}_{k-1}}^t\ga_s\text{d}s)$ is a Wiener process under the law $\mathbb{P}_{k-1}$ of the original branching process, and $\widehat{L}$ is the Radon-Nykodym derivative such that
\begin{equation}\label{eq: r_n derivative}
\frac{\text{d}\widehat{\mathbb{P}}_{k-1}}{\text{d}\mathbb{P}_{k-1}}\Big|_{\widehat{\mathcal{F}}_{t}} =\widehat{L}_t.
\end{equation}
Note that a classical sufficient condition for $L$ to be a martingale is Novikov's condition:
\begin{equation}\label{novikov}
\mathbb{E}\Big[\exp\Big\{\frac{1}{2}\int_{\widehat{T}_{k-1}}^t\Vert  \ga_s \Vert_2^2\text{d}s \Big\} \Big] < \infty.
\end{equation}
Assumption \ref{ass: novikov} ensures that the process $\ga$ is almost surely bounded, and that \eqref{novikov} holds. However this sufficient condition is not necessary to ensure this finite moment hypothesis, and a more detailed analysis on the process dynamics might give a less restrictive set of hypotheses. 
Now we will establish that the martingale $\widehat{L}$ satisfies:
\begin{equation}\label{eq: L egal exp}
    \widehat{L}_t = \frac{\psi_{\phi}(\widehat{Y}_{\widehat{T}_{k-1}},\widehat{\nu}_{\widehat{T}_{k-1}})}{\psi_{\phi}(\widehat{Y}_t,\widehat{\nu}_t)}\exp\left\{\int_{\widehat{T}_{k-1}}^t\frac{G\psi_{\phi}}{\psi_{\phi}}(\widehat{Y}_s,\widehat{\nu}_s)\text{d}s\right\}.
\end{equation}
Using the vectorial formalism introduced earlier and applying Ito's formula we get that 
$$
\ln \psi_{\phi}\left(\mathfrak{e}, \overline{\widehat{\Z}}_t\right) = \ln \psi_{\phi}\left(\mathfrak{e}, \overline{\widehat{\Z}}_{\widehat{T}_{k-1}}\right) + \int_{\widehat{T}_{k-1}}^t\widehat{G}\ln\psi_{\phi}\left(\mathfrak{e}, \overline{\widehat{\Z}}_s\right)\text{d}s + \int_{\widehat{T}_{k-1}}^t\nabla \ln\psi_{\phi}^T\left(\mathfrak{e}, \overline{\widehat{\Z}}_s\right)\ssigma\left(\overline{\widehat{\Z}}_s\right)\text{d}\widehat{\boldsymbol{B}}_s
$$
Using the expression of $\widehat{G}$ introduced in Proposition \ref{prop: generator} we get that
\begin{align*}
\ln \psi_{\phi}\left(\mathfrak{e}, \overline{\widehat{\Z}}_t\right) &= \ln \psi_{\phi}\left(\mathfrak{e}, \overline{\widehat{\Z}}_{\widehat{T}_{k-1}}\right) + \int_{\widehat{T}_{k-1}}^tG\ln\psi_{\phi}\left(\mathfrak{e}, \overline{\widehat{\Z}}_s\right)\text{d}s \\
&\hspace{1.8cm} + \int_{\widehat{T}_{k-1}}^t\frac{\nabla\psi_{\phi}^T}{\psi_{\phi}} \ssigma\ssigma^T\frac{\nabla\psi_{\phi}}{\psi_{\phi}} \left(\mathfrak{e}, \overline{\widehat{\Z}}_s\right)\text{d}s + \int_{\widehat{T}_{k-1}}^t\frac{\nabla\psi_{\phi}^T}{\psi_{\phi}}\left(\mathfrak{e}, \overline{\widehat{\Z}}_s\right)\ssigma\left(\overline{\widehat{\Z}}_s\right)\text{d}\widehat{\boldsymbol{B}}_s\\
&=\ln \psi_{\phi}\left(\mathfrak{e}, \overline{\widehat{\Z}}_{\widehat{T}_{k-1}}\right) + \int_{\widehat{T}_{k-1}}^tG\ln\psi_{\phi}\left(\mathfrak{e}, \overline{\widehat{\Z}}_s\right)\text{d}s + \int_{\widehat{T}_{k-1}}^t\Vert  \ga_s \Vert_2^2\text{d}s - \int_{\widehat{T}_{k-1}}^t\ga_s\text{d}\widehat{\boldsymbol{B}}_s.
\end{align*}
Thus taking the exponential in the previous equation, and using \eqref{eq: L_t}, we get that
\begin{equation*}
    \widehat{L}_t = \frac{\psi_{\phi}(\mathfrak{e}, \overline{\widehat{\Z}}_{\widehat{T}_{k-1}})}{\psi_{\phi}(\mathfrak{e}, \overline{\widehat{\Z}}_t)}\exp\Big\{\int_{\widehat{T}_{k-1}}^tG\ln\psi_{\phi}(\mathfrak{e}, \overline{\widehat{\Z}}_s) + \frac{1}{2}\Vert  \ga_s \Vert_2^2\text{d}s\Big\}.
\end{equation*}
We now detail the expression of the generator $G$ applied to the function $\ln\psi_{\phi}$. Using \eqref{eq: generator} to the function $F_f = \ln\circ\psi_{\phi}$, we have 
\begin{align*}
&G\ln\psi_{\phi}\left(\mathfrak{e},\Zbar\right) = \frac{\partial_1\psi}{\psi_{\phi}}\left(x_{\mathfrak{e}},\langle \nu,\phi \rangle \right)\boldsymbol{a}_{\mathfrak{e}}\left(\Zbar\right)+\frac{1}{2}\Big[\frac{\partial^2_{1,1}\psi}{\psi_{\phi}} - \Big(\frac{\partial_{1}\psi}{\psi_{\phi}}\Big)^2\Big]\left(x_{\mathfrak{e}},\langle \nu,\phi \rangle \right)\left(\ssigma\ssigma^T\right)_{\mathfrak{e},\mathfrak{e}}\left(\Zbar\right) \nonumber\\
& \hspace{1cm}+ \Big[\frac{\partial^2_{1,2}\psi}{\psi_{\phi}} - \frac{\partial_{1}\psi\partial_{2}\psi}{\psi_{\phi}^2}\Big]\left(x_{\mathfrak{e}},\langle \nu,\phi \rangle \right) \int_{\mathcal{X}\times\mathcal{U}} \left(\ssigma\ssigma^T\right)_{u,\mathfrak{e}}\left(\Zbar\right)\phi'\left(x\right) \Zbar(\text{d}x,\text{d}u)\nonumber\\
&\hspace{1cm}+ \frac{\partial_2\psi}{\psi_{\phi}}\left(x_{\mathfrak{e}},\langle \nu,\phi \rangle \right) \int_{\mathcal{X}\times\mathcal{U}}\boldsymbol{a}_u\left(\Zbar\right)\phi'\left(x\right) + \frac{\left(\ssigma\ssigma^T\right)_{u,u}}{2}\left(\Zbar\right)\phi''\left(x\right)\Zbar(\text{d}x,\text{d}u)\nonumber\\
&\hspace{1cm}+\frac{1}{2}\Big[\frac{\partial^2_{2,2}\psi}{\psi_{\phi}} - \left(\frac{\partial_{2}\psi}{\psi_{\phi}}\right)^2\Big]\left(x_{\mathfrak{e}},\langle \nu,\phi \rangle \right)\times\int_{\mathcal{X}\times\mathcal{U}}\int_{\mathcal{X}\times\mathcal{U}}\left(\ssigma\ssigma^T\right)_{u,v}\left(\Zbar\right)\phi'\left(x\right)\phi'\left(y\right)\Zbar(\text{d}x,\text{d}u)\Zbar(\text{d}y,\text{d}v).
\end{align*}    
First we recognize the expression of $G\psi_{\phi}$ in this expression
\begin{align*}
&G\ln\psi_{\phi}\left(\mathfrak{e},\Zbar\right) = \frac{G\psi_{\phi}}{\psi_{\phi}}\Big(\mathfrak{e},\Zbar\Big)
-\frac{1}{2}\left(\frac{\partial_{1}\psi}{\psi_{\phi}}\right)^2\left(x_{\mathfrak{e}},\langle \nu,\phi \rangle \right)\left(\ssigma\ssigma^T\right)_{\mathfrak{e},\mathfrak{e}}\left(\Zbar\right) \nonumber\\
&\hspace{1cm}- \frac{\partial_{1}\psi\partial_{2}\psi}{\psi_{\phi}^2}\left(x_{\mathfrak{e}},\langle \nu,\phi \rangle \right) \int_{\mathcal{X}\times\mathcal{U}} \left(\ssigma\ssigma^T\right)_{u,\mathfrak{e}}\left(\Zbar\right)\phi'\left(x\right) \Zbar(\text{d}x,\text{d}u)\nonumber\\
&\hspace{1cm}-\frac{1}{2} \Big(\frac{\partial_{2}\psi}{\psi_{\phi}}\Big)^2\left(x_{\mathfrak{e}},\langle \nu,\phi \rangle \right)\times\int_{\mathcal{X}\times\mathcal{U}}\int_{\mathcal{X}\times\mathcal{U}}\left(\ssigma\ssigma^T\right)_{u,v}\left(\Zbar\right)\phi'\left(x\right)\phi'\left(y\right)\Zbar(\text{d}x,\text{d}u)\Zbar(\text{d}y,\text{d}v),
\end{align*} 
then we recognize the expression of $\nabla\psi_{\phi}^T \ssigma\ssigma^T\nabla\psi_{\phi}$:
\begin{align*}
G\ln\psi_{\phi}\left(\mathfrak{e},\Zbar\right) &= \frac{G\psi_{\phi}}{\psi_{\phi}}\left(\mathfrak{e},\Zbar\right)
-\frac{1}{2}\frac{1}{\psi_{\phi}^2}\nabla\psi_{\phi}^T \ssigma\ssigma^T\nabla\psi_{\phi}\left(\mathfrak{e},\Zbar\right).
\end{align*} 
Then, using the definition of $\boldsymbol{\gamma}$ and applying again Ito's formula on the previous equation yield
$$
G\psi_{\phi}\left(\mathfrak{e}, \overline{\widehat{\Z}}_t\right) = G\ln\psi_{\phi}\left(\mathfrak{e}, \overline{\widehat{\Z}}_t\right) + \frac{1}{2}\Vert  \ga_t \Vert_2^2.
$$
This last equation gives \eqref{eq: L egal exp}, and the change of measure \eqref{eq: r_n derivative} concludes the proof.
\end{proof}

\subsection{Proofs of Section \ref{section: example}}\label{section: proof of section 2}
This section is devoted to the proofs concerning the colonial population model at the heart of the article.
\subsubsection{Proof of Theorem \ref{thm: sampling colony}}
We introduce the sets $E_0:=\mathcal{U}\times\mathbb{R}_+$ and $E_2:=\mathcal{U}\times\mathbb{R}_+\times[0,1]$. 
We introduce the Poisson Point measures $Q_2$ and $Q_0$ respectively on the sets $[0,t] \times E_2$ and $[0,t] \times E_0$ and of respective intensity measures $\text{d}sn(\text{d}u)\text{d}rq(\text{d}\theta)$ and $\text{d}sn(\text{d}u)\text{d}r$, where $n(\text{d}u)$ denotes the counting measure on $\mathcal{U}$ and $q$ is the probability measure giving the sharing of resources at a spliting event.
The measure-valued process $(\nu_t, t\geq 0)$ following the branching mechanisms previously introduced and where the colonies traits follow the SDE \eqref{eq: eds indiv}, is a weak-solution of the following SDE, for all function $F \in \mathcal{C}^2(\mathbb{R}_+,\mathbb{R})$:
\begin{multline}\label{eq: eds pop}
\langle \nu_t,F\rangle = \langle \nu_0,F \rangle + M_F(t) + \int_0^T\int_{\mathbb{R}_+}\left(axF'(x)+ \frac{\sigma^2 x^2 }{2}F''(x)\nu_s(\text{d}x)\right)\text{d}s \nonumber \\
+ \int_{[0,t] \times E_2} \mathbbm{1}_{\Big\{u \in \mathbb{G}_s,r \leq \lambda \frac{X^u_{s-}}{\langle \nu_{s-},I_d\rangle}\Big\}}\Big(F\left(\theta X^u_{s-}\right) + F\left(\left(1-\theta\right)X^u_{s-}\right) -F\left(X^u_{s-}\right)\Big) Q_2(\text{d}s,\text{d}u,\text{d}r,\text{d}\theta) \nonumber\\
-\int_{[0,t] \times E_0} \mathbbm{1}_{\left\{u \in \mathbb{G}_{s-}\right\}}\mathbbm{1}_{\left\{r \leq \mu\right\}}F\left(X^u_{s-}\right)Q_0(\text{d}s,\text{d}u,\text{d}r).
\end{multline}
The martingale $M_F(t)$ is given by
$$
M_F(t) := \int_0^T\cdots \int_0^T\sum_{u\in\mathbb{G}(t)}\partial_xF\left(X^u_s\right) \frac{\sigma}{\sqrt{1+\delta^2}} X^u_s \left(\text{d}B^u_s + \delta\text{d}W_s\right).
$$
The first line is related to the traits dynamics of the colonies between the branching events. The second line gives the fission events and the extinction events are encompassed in the last line. Notice that this process does not explode in finite time as the branching rates are bounded and the mass is conserved at birth, see \textit{e.g.} \cite{marguet19}. 

Without loss of generality, we rearrange the labels in the lexicographic order to fit the vectorial formalism introduced in Section \ref{section: general spine}. Thus we have 
$$
\boldsymbol{X}_t :=  \left(X^u_t, u\in\mathbb{G}(t)\right) = \left(X^1_t, \cdots, X^{N_t}_t\right) \quad \text{and} \quad \boldsymbol{a}(\Zbar_t) := \left(aX^1_t, \cdots, aX^{N_t}_t \right).
$$ 
The vector of Brownian processes is then $\boldsymbol{B}_t := (B^1_t, \cdots, B^{N_t}_t,  W_t)$ and the diffusion matrix
\begin{equation}\label{diff matrix}
      \ssigma(\Zbar_t)  :=
  \left[ {\begin{array}{ccccc}
    \frac{\sigma}{\sqrt{1+\delta^2}} X^1_t & 0 & \cdots & 0 & \frac{\sigma}{\sqrt{1+\delta^2}} \delta X^1_t\\
    0 & \frac{\sigma}{\sqrt{1+\delta^2}} X^2_t & \ddots & \vdots & \frac{\sigma}{\sqrt{1+\delta^2}} \delta X^2_t\\
    \vdots & \ddots & \ddots & 0 & \vdots\\
    0 & \cdots & 0 & \frac{\sigma}{\sqrt{1+\delta^2}} X^{N_t}_t & \frac{\sigma}{\sqrt{1+\delta^2}} \delta X^{N_t}_t\\
  \end{array} } \right].  
\end{equation}
With this formalism, the dynamics of the traits between jumps satisfy Equation \eqref{eq: eds trait}.
From the dynamics describing the colonial process in Section \ref{section: Many-to-One formula}, the operator $\mathcal{G}$ introduced in \eqref{def: mathcal G} is defined for all measurable function $\psi_{\phi}(\mathfrak{e},\Zbar) = \psi(x_{\mathfrak{e}}) $ by
\begin{multline*}
\mathcal{G}\psi_{\phi}(\mathfrak{e},\Zbar) := \psi'\left(x_{\mathfrak{e}}\right)a x_{\mathfrak{e}}  +  \frac{\lambda x_{\mathfrak{e}}}{\langle \Z,Id\rangle}\int_{\mathcal{X}^n} \left[  \psi\left(\theta x_{\mathfrak{e}}\right) + \psi\left((1-\theta) x_{\mathfrak{e}}\right) 
- \psi\left(x_{\mathfrak{e}}\right) \right] q(\text{d}\theta)  - \mu \psi(x_{\mathfrak{e}}).
\end{multline*}
Then, as the fission event is mass conservative, the bias function $\psi(x_{\mathfrak{e}}) = x_{\mathfrak{e}}$ is an eigenfunction of the operator $\mathcal{G}$ for the eigenvalue $(a-\mu)$.
To construct the spine process $((\widehat{Y}_t,\Zhat_t), t\geq 0)$, we use the non-negative function $\psi$ defined for all  $x \in \mathbb{R}_+^*$ as $\psi(x) = x$. 
Using \eqref{diff matrix}, we derive the expression of the additional drift term involved in the SDE that describes the evolution of the traits in the $\psi$-spine process between two branching events
$$
\frac{\ssigma\ssigma^T(\overline{\widehat{\nu}}_t) \nabla \psi(E_t,\overline{\widehat{\nu}}_t)}{\widehat{Y}_t} = \Big( \sigma^2 \widehat{Y}_t, \Big(\sigma^2 \frac{\delta^2}{1 + \delta^2} \widehat{X}^u_t, u\in\widehat{\mathbb{G}}(t)\backslash\{E_t\}\Big)\Big).
$$
From this expression, we obtain the SDEs \eqref{eq: eds hors spine} and \eqref{eq: eds spine}.
The rates of jumps in the spine process are given by the expression of the generator of the process given by Proposition \ref{prop: generator}, applied with the chosen function $\psi$. We obtain that the branching rates in the spine process are the same as the ones in the colonial process. The choice of the new spinal individual follows the law \eqref{eq: init spine} and \eqref{eq: next spine choice}. The death rates in the spine process are the same as the ones in the colonial process for the non spinal individuals ($\mu$), and the spine does not extinct.
Finally, applying Theorem \ref{thm: sampling} for $U_t$ an individual picked uniformly at random in $\mathbb{G}(t)$ and for all measurable function $F$ on the set of càdlàg trajectories on $\mathbb{R}_+^*\times\mathbb{U}$, we obtain Theorem \ref{thm: sampling colony}.

\subsubsection{Study of the spinal process} \label{section_spinal}
In order to derive results on the original population process, we first need to establish some technical lemmas on the spinal process.
The first step is to use a set of autonomous processes whose law will well describe the law of the set of processes $ ((\widehat{N}_t,\widehat{R}_t,\widehat{Y}_t,(X_t^u,u \in\widehat{\mathbb{G}}_{t}\backslash\{E_t\})),t \geq 0) $.
To state the following result, we introduce the fraction of resources contained in the spine $\widehat{Z}_t$
as well as in an individual $u \in\widehat{\mathbb{G}}_{t}\backslash\{E_t\}$, $\widehat{Z}^u_{t}$ via
$$ \widehat{Z}_t :={\widehat{Y}_t}/{\widehat{R}_t}\quad \text{and} \quad \widehat{Z}^u_{t}:={\widehat{X}^u_{t}}/{\widehat{R}_t}. $$
\begin{lemma} \label{prop_autonomous}
The process $ (\mathfrak X_t, t \geq 0):=((\widehat{N}_t,\widehat{Z}_t,(\widehat{Z}_t^u,u \in\widehat{\mathbb{G}}_{t}\backslash\{E_t\})),t \geq 0) $ 
\begin{enumerate}
    \item is an autonomous and regenerative process
    \item does not depend on the parameter $a$
    \end{enumerate}
\end{lemma}
Recall that a regenerative process is a process $(\mathfrak X_t, t \geq 0)$ such that there exists a renewal process with epochs $S_1,S_2, ...$ such that for all $n \geq 1$,
$$((\mathfrak X_{t+S_n}, t\geq 0 ),(S_{m+n}-S_n, m\in \N)) \overset{d}{=}((\mathfrak X_{t+S_1}, t\geq 0 ),(S_{m+1}-S_1,m\in \N))$$
and $((\mathfrak X_{t+S_n}, t\geq 0 ),(S_{m+n}-S_n,m\in \N))$ is independent of $ ((\mathfrak X_t, t \in [0,S_n)),S_1, S_2,...,S_n).$
Such a process admits long term properties which will be helpful in the study of the spinal process. Lemma \ref{prop_autonomous} will be a key tool to study the long time behavior of the fraction of resources $\widehat Z$ contained in the spinal individual. 
\begin{proof}
From \eqref{prop: generator} with the bias function chosen in Section \ref{section: spine construct ex}, we get that the trait $\widehat{Y}_t$ of the spinal individual is the solution of the following SDE
\begin{align}\label{eq: eds indiv spine 2}
&\widehat{Y}_t = \widehat{Y}_0 + \int_0^t\left(a + \sigma^2 \right)\widehat{Y}_s\text{d}s + \int_0^t\frac{\sigma}{\sqrt{1+\delta^2}} \widehat{Y}_s\left(\text{d}\widehat{B}^*_s + \delta\text{d}\widehat{W}_s\right) \\
&-\int_0^t\int_{\mathbb{R}_+}\int_0^1\mathbbm{1}_{\{ r \leq  \frac{\lambda \theta\widehat{Y}_{s-}}{\langle \widehat{\nu}_{s-},I_d\rangle}\}}\left(1-\theta\right)\widehat{Y}_sQ_1^*(\text{d}s,\text{d}r,\text{d}\theta)
-\int_0^t\int_{\mathbb{R}_+}\int_0^1\mathbbm{1}_{\{ r \leq\frac{ \lambda \left(1-\theta\right)\widehat{Y}_{s-}}{\langle \widehat{\nu}_{s-},I_d\rangle}\}}\theta \widehat{Y}_sQ_2^*(\text{d}s,\text{d}r,\text{d}\theta)\nonumber
\end{align} 
where $\widehat{B}^*_s := \widehat{B}^{E_s}_s$ is a Wiener process and $Q_1^*,Q_2^*$ are independent Poisson Point Measures on $\mathbb{R}_+\times\mathbb{R}_+\times[0,1]$ of intensity $\text{d}s\text{d}rq(\text{d}\theta)$.
For all individuals $u$ outside the spine $E_t$ of resources $\widehat{Y}_t$, i.e. for all $u\in\widehat{\mathbb{G}}(t)\backslash \{E_t\}$, between the branching events, we denote $\widehat{X}^u_t$ the trait of the individual of label $u$ living at time $t$ and with a slight abuse of notation it also denotes the trait of its ancestor at times $s<T_{\text{birth}}^u$. Thus, for all $u \in \widehat{\mathbb{G}}(t)\backslash \{E_t\}$ we can write the evolution of the trait of the 'lineage' of $u$ that survived until time $t$ as follows
\begin{multline}\label{eq: eds indiv not spine 2}
\widehat{X}^u_t = \widehat{X}^u_0 + \int_0^t\left(a\widehat{X}^u_s + \sigma^2\frac{\delta^2}{1+\delta^2} \widehat{X}^u_t \right)\text{d}s + \int_0^t\frac{\sigma}{\sqrt{1+\delta^2}} \widehat{X}^u_s\left(\text{d}\widehat{B}^u_s + \delta\text{d}\widehat{W}_s\right) \\ -\int_0^t\int_{\mathbb{R}_+}\int_0^1\mathbbm{1}_{\{ r \leq \lambda \frac{\widehat{X}^u_{s-}}{\widehat{R}_{s-}}\}}\left(1-\theta\right)Q^u(\text{d}s,\text{d}r,\text{d}\theta)
\end{multline} 
with $Q^u$ a Poisson Point Measure on $\mathbb{R}_+\times\mathbb{R}_+\times[0,1]$ of intensity $\text{d}s\text{d}rq(\text{d}\theta)$.
From \eqref{eq: eds indiv spine 2} and \eqref{eq: eds indiv not spine 2} we obtain that the total resources $\widehat{R}_t$ in the spinal population follows the SDE
\begin{multline}\label{eq: sde total biomass 1}
\widehat{R}_t = \widehat{R}_0 + \int_0^t \Big(a\widehat{R}_s + \frac{\sigma^2}{1+\delta^2}\left( \widehat{Y}_s + \delta^2\widehat{R}_s\right)\Big)\text{d}s + \int_0^t\frac{\sigma}{\sqrt{1+\delta^2}} \Big(\sum_{u\in\widehat{G}_s}\widehat{X}^u_s\text{d}\widehat{B}^u_s +\delta \widehat{R}_s\text{d}\widehat{W}_s\Big)\\
- \int_0^t\int_{\mathcal{U}}\int_{\mathbb{R}_+}\mathbbm{1}_{\left\{u \in\widehat{\mathbb{G}}_{s-}\backslash\{E_t\}\right\}}\mathbbm{1}_{\left\{r \leq \mu\right\}}\widehat{X}^u_{s-}\widehat{Q}_0(\text{d}s,\text{d}u,\text{d}r).
\end{multline}
We want to establish SDEs for $\widehat{X}^u_t/\widehat{R}_t$ for all individual $u\in\widehat{\G}(t)$. To this end we need to differentiate an individual in \eqref{eq: sde total biomass 1}. For all $\e \in \mathbb{G}(t)$
\begin{align}\label{eq: sde total biomass}
\widehat{R}_t &= \widehat{R}_0 + \int_0^t \Big(a\widehat{R}_s + \frac{\sigma^2}{1+\delta^2}\left( \widehat{Y}_s + \delta^2\widehat{R}_s\right)\Big)\text{d}s + \int_0^t\frac{\sigma}{\sqrt{1+\delta^2}}\Big(\sum_{u\in\widehat{G}_s\backslash\{\e\}}\left(\widehat{X}^u_s\right)^2\Big)^{1/2}\text{d}\widehat{B}_s \\ &+ \int_0^t\frac{\sigma}{\sqrt{1+\delta^2}} \left(\widehat{X}^{\e}_s\text{d}\widehat{B}^{\e}_s +\delta \widehat{R}_s\text{d}\widehat{W}_s\right)
- \int_0^t\int_{\mathcal{U}}\int_{\mathbb{R}_+}\mathbbm{1}_{\left\{u \in\widehat{\mathbb{G}}_{s-}\backslash\{E_{s-}\}\right\}}\mathbbm{1}_{\left\{r \leq \mu\right\}}\widehat{X}^u_{s-}\widehat{Q}_0(\text{d}s,\text{d}u,\text{d}r)\nonumber
\end{align}
with $\widehat{Q}_0$ a Poisson Point Measure on $\mathbb{R}_+\times \mathcal{U}\times\mathbb{R}_+$ of intensity $\text{d}sn(\text{d}u)\text{d}r$, and 
$$\Big(\sum_{u\in\widehat{G}_s\backslash\{\e\}}\left(\widehat{X}^u_s\right)^2\Big)^{1/2} \overset{d}{=} \sum_{u\in\widehat{G}_s\backslash\{\e\}}\widehat{X}^u_s\text{d}\widehat{B}^u_s.$$
Finally the size $\widehat{N}_t$ of the spinal population is given by:
\begin{equation}\label{EDS_widehatN}
    \widehat{N}_t = \widehat{N}_0 + \int_0^t\int_{\mathbb{R}_+}\mathbbm{1}_{\left\{r \leq \lambda\right\}}\widehat{Q}'_2(\text{d}s,\text{d}r) - \int_0^t\int_{\mathbb{R}_+}\mathbbm{1}_{\left\{r \leq \mu\left(\widehat{N}_{s-}-1\right)\right\}}\widehat{Q}'_0(\text{d}s,\text{d}r).
\end{equation}
Note that $\widehat{Q}'_0$ is equal to the measure $\widehat{Q}_0$, introduced in \eqref{eq: sde total biomass}, integrated over the individuals.
To prove the first point of Proposition \ref{prop_autonomous}, we will apply Itô formula to the function $\phi(y,r) := y/r$, using \eqref{eq: eds indiv spine 2} and \eqref{eq: sde total biomass} with $\e = E_t$ to obtain the SDE satisfied by the processes $\widehat{Z}$. Recall that $\widehat{B}^* = \widehat{B}^{E_t}$ is the Wiener process associated with the spinal individual. We get
\begin{multline*}
   d \widehat{Z}_t= \left(\left(a+\sigma^2\right)\widehat{Y}_t\partial_1\phi + \left(a\widehat{R}_s + \frac{\sigma^2}{1+\delta^2} \widehat{Y}_s +\sigma^2\frac{\delta^2}{1+\delta^2}\widehat{R}_s\right)\partial_2\phi\right) \text{d}t \\
    + \frac{1}{2}\Bigg(\Sigma^s_{1,1}\partial^2_{1,1}\phi + 2\Sigma^s_{1,2}\partial^2_{1,2}\phi+\Sigma^s_{2,2}\partial^2_{2,2}\phi \Bigg)\text{d}t 
    + \frac{\sigma\delta}{\sqrt{1+\delta^2}} \left(\widehat{Y}_t\partial_{1}\phi +\widehat{R}_t \partial_{2}\phi\right)\text{d}\widehat{W}_t\\
    + \frac{\sigma}{\sqrt{1+\delta^2}} \left(\widehat{Y}_t\left(\partial_{1}\phi + \partial_{2}\phi\right)\text{d}\widehat{B}^*_t +\Big(\sum_{u\in\widehat{G}_s\backslash\{\e\}}\left(\widehat{X}^u_s\right)^2\Big)^{1/2}\partial_{2}\phi \text{d}\widehat{B}_t \right) \\   
    -\int_{\mathbb{R}_+}\int_0^1\mathbbm{1}_{\big\{ r \leq \lambda \theta \frac{\widehat{Y}_{t\minus}}{\widehat{R}_{t\minus}}\big\}}\left(1-\theta\right)\frac{\widehat{Y}_{t\minus}}{\widehat{R}_{t\minus}}Q_1^*(\text{d}t,\text{d}r,\text{d}\theta)
-\int_{\mathbb{R}_+}\int_0^1\mathbbm{1}_{\big\{ r \leq \lambda \left(1-\theta\right)\frac{\widehat{Y}_{t\minus}}{\widehat{R}_{t\minus}}\big\}}\theta \frac{\widehat{Y}_{t\minus}}{\widehat{R}_{t\minus}}Q_2^*(\text{d}t,\text{d}r,\text{d}\theta)\\
+ \int_{\mathcal{U}}\int_{\mathbb{R}_+}\mathbbm{1}_{\left\{u \in\widehat{\mathbb{G}}_{t\minus}\backslash\{E_{t\minus}\}\right\}}\mathbbm{1}_{\left\{r \leq \mu\right\}}\frac{\widehat{Y}_{t\minus}}{\widehat{R}_{t\minus}}\frac{\widehat{X}^u_{t\minus}}{\widehat{R}_{t\minus}-\widehat{X}^u_{t\minus}}\widehat{Q}_0(\text{d}t,\text{d}u,\text{d}r)
\end{multline*}
with
\begin{equation*}
    \Sigma^s := 
\renewcommand*{\arraystretch}{1.5}
\begin{pmatrix}
\sigma^2 \widehat{Y}_s^2 & \frac{\sigma^2}{1+\delta^2}\widehat{Y}_s \left( \widehat{Y}_s +  \delta^2 \widehat{R}_s\right)  \\
\frac{\sigma^2}{1+\delta^2}\widehat{Y}_s \left( \widehat{Y}_s +  \delta^2 \widehat{R}_s\right)  & \frac{\sigma^2}{1+\delta^2} \left( \sum_{u \in \widehat{\mathbb{G}}_s}\left(\widehat{X}^u_s\right)^2 +\delta^2 \widehat{R}_s^2\right)
\end{pmatrix},
\end{equation*}
and
\begin{equation*}
\partial_1\phi = {1}/{\widehat{R}_s}, \quad \partial_2\phi = -{\widehat{Y}_s}/{\widehat{R}_s^2} \quad \partial^2_{1,1}\phi = 0, \quad \partial^2_{1,2}\phi = -{1}/{\widehat{R}_s^2}, \quad  \partial^2_{2,2}\phi = {2\widehat{Y}_s}/{\widehat{R}_s^3}.
\end{equation*}
 From the previous computation, we obtain
\begin{align}\label{EDS_Z} 
 &   \widehat{Z}_t = \widehat{Z}_0 + \int_0^t\frac{\sigma^2\widehat{Z}_s}{1+\delta^2}\left(1-2\widehat{Z}_s + \sum_{u \in\widehat{\mathbb{G}}_{s}}\left(\widehat{Z}^u_s\right)^2 \right)\text{d}s 
    + \int_0^t\frac{\sigma\widehat{Z}_s }{\sqrt{1+\delta^2}} \left(1-\widehat{Z}_s \right)\text{d}\widehat{B}^*_s \\&- \int_0^t\frac{\sigma\widehat{Z}_s}{\sqrt{1+\delta^2}}\Big(\sum_{u\in\widehat{G}_s\backslash\{\e\}}\left(\widehat{X}^u_s\right)^2\Big)^{1/2}\text{d}\widehat{B}_s  
    -\int_0^t\int_{\mathbb{R}_+}\int_0^1\mathbbm{1}_{\left\{ r \leq \lambda \theta \widehat{Z}_{s-}\right\}}\left(1-\theta\right)\widehat{Z}_{s-}Q_1^*(\text{d}s,\text{d}r,\text{d}\theta)\nonumber\\
&-\int_0^t\int_{\mathbb{R}_+}\int_0^1\mathbbm{1}_{\{ r \leq \lambda \left(1-\theta\right)\widehat{Z}_{s-}\}}\theta \widehat{Z}_{s-}Q_2^*(\text{d}s,\text{d}r,\text{d}\theta)
+ \int_0^t\int_{\mathcal{U}}\int_{\mathbb{R}_+}\mathbbm{1}_{\left\{u \in\widehat{\mathbb{G}}_{s-}\backslash\{E_{s-}\},r \leq \mu\right\}}\frac{\widehat{Z}_{s-}\widehat{Z}^u_{s-}}{1-\widehat{Z}^u_{s-}}\widehat{Q}_0(\text{d}s,\text{d}u,\text{d}r).\nonumber
\end{align}
Similarly we obtain for all $u \in \widehat{\mathbb{G}}_{t\minus}\backslash\{E_{t\minus}\}$
\begin{align}\label{EDS_Zu} 
&    \widehat{Z}^u_t = \widehat{Z}_0 + \int_0^t\frac{\sigma^2\widehat{Z}^u_s}{1+\delta^2}\left(\sum_{u \in\widehat{\mathbb{G}}_{s}}\left(\widehat{Z}^u_s\right)^2-\widehat{Z}^u_s-\widehat{Z}_s \right)\text{d}s 
    + \int_0^t\frac{\sigma \widehat{Z}^u_s}{\sqrt{1+\delta^2}} \left(1-\widehat{Z}^u_s \right)\text{d}\widehat{B}^u_s \\
    &- \int_0^t\frac{\sigma\widehat{Z}^u_s}{\sqrt{1+\delta^2}} \Big(\sum_{v \in\widehat{\mathbb{G}}_{s}\backslash\{u\}}\left(Z^v_s\right)^2\Big)^{1/2}\text{d}\widehat{B}_s  
    -\int_0^t\int_{\mathbb{R}_+}\int_0^1\mathbbm{1}_{\left\{ r \leq \lambda \theta \widehat{Z}^u_{s-}\right\}}\left(1-\theta\right)\widehat{Z}^u_{s-}Q_1^*(\text{d}s,\text{d}r,\text{d}\theta)-\nonumber\\
&\int_0^t\int_{\mathbb{R}_+}\int_0^1\mathbbm{1}_{\{ r \leq \lambda \left(1-\theta\right)\widehat{Z}^u_{s-}\}}\theta \widehat{Z}^u_{s-}Q_2^*(\text{d}s,\text{d}r,\text{d}\theta)+ \int_0^t\int_{\mathcal{U}}\int_{\mathbb{R}_+}\mathbbm{1}_{\{v \in\widehat{\mathbb{G}}_{s-}\backslash\{E_{s-},u\},r \leq \mu\}}\frac{\widehat{Z}^u_{s-}Z^v_{s-}}{1-Z^v_{s-}}\widehat{Q}_0(\text{d}s,\text{d}u,\text{d}r).\nonumber
\end{align}
Finally, the cardinal of $\widehat{\mathbb{G}}$ is given by $\widehat{N}$. The latter is autonomous and satifies SDE \eqref{EDS_widehatN}.
Hence, from SDEs \eqref{EDS_widehatN}, \eqref{EDS_Z} and \eqref{EDS_Zu} we obtain that $ \mathfrak X $ is an autonomous process. To prove that it is also a regenerative process, it is enough to notice that when $\widehat{N}_t=1$, then $\widehat Z_t=1$ and $\widehat{\mathbb{G}}_{t}\backslash\{E_{t}\} = \emptyset $, hence in this case $ \mathfrak X_t= (1,1). $ 
Moreover, $S_1<\infty$ almost surely, and the successive durations $(S_2-S_1,S_3-S_2,...)$ are non-lattice iid random variables with finite expectation. It ends the proof of point $1.$ of Proposition \ref{prop_autonomous}. To prove point $2.$ it is enough to notice that the parameter $a$ does not appear in Equations \eqref{EDS_widehatN}, \eqref{EDS_Z} and \eqref{EDS_Zu}.
\end{proof}

A natural question is to determine how the parameters $(\lambda,\mu,\sigma,\delta)$ influence the long time behavior of the system. Obtaining quantitative bounds is delicate. We may however get some indications, stated in Lemma \ref{prop_lim_frac}. 
Recall \eqref{def: alpha beta} for the definition of $\alpha$ and $\beta$.

\begin{lemma} \label{prop_lim_frac}
The random variable $\widehat Z_t$ converges in law to a random variable $\widehat Z_\infty$ as $t \to \infty$. 
    \begin{enumerate}
        \item For every bounded continuous function on $\N^* \times (0,1] \times [0,1]^{\N^*}$,
    \begin{equation}\label{mathfrakX} 
    \E[f(\widehat Z_\infty)] = \lim_{t \to \infty} \frac{1}{t} \int_0^t f(\widehat Z_s)\text{d}s
    = \lim_{t \to \infty} \frac{1}{t} \int_0^t \E[f(\widehat Z_s)]\text{d}s.  
    \end{equation}
    \item Let us denote $M := \left( \sqrt{1+ 8\alpha\left( 1 +\frac{8}{27}\beta\right)} - 1 \right)/4\alpha$.
       $$ \E[\widehat Z_\infty] \leq \frac{2\beta - 1}{4(\alpha + \beta)}\left( 1+ sign(2\beta-1)\sqrt{1+ \frac{8(\alpha + \beta)}{(2\beta - 1)^2}} \right) \wedge M<1. $$
    \item 
    $$ \E[\widehat Z_\infty] \geq  \left(\frac{\beta+ \beta\frac{\mu}{\lambda} \left(1- e^{-\lambda/\mu}\right)+2}{2\beta- \frac{4\lambda}{\mu}\E[\Theta \ln\Theta]+ 2}\right) \vee  \frac{\mu}{\lambda}\left(1- e^{-\lambda/\mu}\right).  $$
    \end{enumerate}
\end{lemma}
Notice that the upper bound has a limit ($1/\sqrt{2\alpha+1}$) as $2\beta \to 1$.

\begin{proof}
    The first point is a direct application of Theorems 7.1.4 and 7.1.5 of \cite{bladt2017matrix}. It is a consequence of the regenerative property of the process $\mathfrak X$, stated in Proposition \ref{prop_autonomous}.
  Let us now focus on point $(2)$, which states an upper bound for the mean fraction of resources contained in the spine in large time. Let us introduce $K_{\theta} := \mathbb{E}[\Theta(1-\Theta)]$. From \eqref{EDS_Z}, we deduce the SDE
  \begin{multline}\label{eq: eds z}
   \widehat{Z}_t = \widehat{Z}_0 + M_t 
   +\int_0^t\widehat{Z}_s\Big[\frac{\sigma^2}{1+\delta^2}\Big(1-2\widehat{Z}_s + \sum_{u \in\widehat{\mathbb{G}}_{s}}(\widehat{Z}^u_s)^2\Big)    +  \mu\sum_{u \in\widehat{\mathbb{G}}_{s}\backslash\{E_s\}}\frac{\widehat{Z}^u_{s}}{1-\widehat{Z}^u_s} -2\widehat{Z}_s\lambda K_{\theta}\Big]\text{d}s,
\end{multline}
where $M_t$ is a square integrable martingale.
Notice that almost surely, for all $u\in\widehat{\mathbb{G}}_s\backslash\{E_s\}$, $ \widehat{Z}^u_s \leq 1 - \widehat{Z}_s$ and the second sum in \eqref{eq: eds z} can be upper bounded as follows
$$
\sum_{u \in\widehat{\mathbb{G}}_{s\minus}\backslash\{E_t\}}\frac{\widehat{Z}^u_{s}}{1-\widehat{Z}^u_s} \ \leq  \ \frac{1-\widehat{Z}_s}{\widehat{Z}_s}.
$$
Hence,
\begin{align*}
     \widehat{Z}_t &\leq \widehat{Z}_0 + \int_0^t\widehat{Z}_s\left[2\frac{\sigma^2}{1+\delta^2}\left(1-\widehat{Z}_s\right)    +  \mu\frac{1-\widehat{Z}_s}{\widehat{Z}_s}-2\widehat{Z}_s\lambda K_{\theta}\right]\text{d}s + M_t \nonumber \\ \label{maj_Z}
     &=  \widehat{Z}_0 + \mu \int_0^t\left(1+ \widehat{Z}_s\left[2\beta -  1 -2(\alpha + \beta)\widehat{Z}_s\right]\right)\text{d}s + M_t 
\end{align*}
with the definition of $\alpha$ and $\beta$ in \eqref{def: alpha beta}.
From Jensen inequality, we obtain
$$ \frac{\text{d} \E[\widehat{Z}_t] }{\text{d}t}\leq \mu\left( 1 +  \E[\widehat{Z}_t]\left[2\beta -  1 - 2(\alpha +\beta) \E[\widehat{Z}_t]\right]\right). $$
    Solving this second order polynomial we obtain the first upper bound of Lemma \ref{prop_lim_frac} point $(2)$.
The second upper bound, $M$, is obtained using that:
\begin{equation*}
    \Big(\sum_{u \in\widehat{\mathbb{G}}_{s}}\widehat{X}^u_s\Big)^2 \geq \sum_{u \in\widehat{\mathbb{G}}_{s}}\left(\widehat{X}^u_s\right)^2 + 2 \widehat{Y}_s \sum_{u \in\widehat{\mathbb{G}}_{s}\backslash\{E_s\}}\widehat{X}^u_s
\end{equation*}
that gives an upper bound of the sum of ratios
\begin{equation*}
    \sum_{u \in\widehat{\mathbb{G}}_{s}}(\widehat{Z}^u_s)^2 \leq 1 - 2\widehat{Z}_s\left(1-\widehat{Z}_s\right).
\end{equation*}
Using this refined upper bound we have that 
\begin{align*}
    \frac{\text{d} \E[\widehat{Z}_t] }{\text{d}t} &\leq \mu +  \left(2\sigma^2-  \mu \right)\E[\widehat{Z}_t] -2(\sigma^2+\lambda K_{\theta})\E[\widehat{Z}_t^2] -2\sigma^2\E[\widehat{Z}^2_s\left(1-\widehat{Z}_s\right)] \nonumber \\
    &\leq \mu + \left(2\sigma^2-  \mu \right)\E\left[\widehat{Z}_t\right] -2\lambda K_{\theta}\E[\widehat{Z}_t^2] -2\sigma^2\E[2\widehat{Z}^2_s-\widehat{Z}^3_s]\\
    &\leq \mu+\frac{8}{27}\sigma^2  -\mu \E[\widehat{Z}_t] -2\lambda K_{\theta}\E[\widehat{Z}_t^2] \nonumber
\end{align*}    
where we used that for $x\in[0,1]$, $2x^2-x^3 \geq x - b $ with a constant $b$ such that $b := \max_{[0,1]}x(x-1)^2 = 4/27$. This entails that $   \limsup_{t \to \infty} \E[\widehat{Z}_t]\leq M$ and concludes the proof of point $(2)$.

    Let us now focus on point $(3)$. We aim at finding a lower bound for the mean fraction of resources contained in the spine in large time. From \eqref{EDS_Z} and Itô formula we get
\begin{multline*} 
   \ln \frac{\widehat{Z}_t}{\widehat{Z}_0} =\int_0^t\frac{\sigma^2}{2(1+\delta^2)} \left[ \Big(1-\widehat{Z}_s \right)^2+\sum_{u \in\widehat{\mathbb{G}}_{s}\backslash \{E_t\}}\left(\widehat{Z}^u_s\right)^2\Big]\text{d}s  
    \\
    +\int_0^t\int_{\mathbb{R}_+}\int_0^1\mathbbm{1}_{\{ r \leq \lambda \theta \widehat{Z}_{s-}\}}\ln\theta Q_1^*(\text{d}s,\text{d}r,\text{d}\theta)
+\int_0^t\int_{\mathbb{R}_+}\int_0^1\mathbbm{1}_{\{ r \leq \lambda \left(1-\theta\right)\widehat{Z}_{s-}\}}\ln (1-\theta)Q_2^*(\text{d}s,\text{d}r,\text{d}\theta)\\
- \int_0^t\int_{\mathcal{U}}\int_{\mathbb{R}_+}\mathbbm{1}_{\left\{u \in\widehat{\mathbb{G}}_{s-}\backslash\{E_{s-}\}\right\}}\mathbbm{1}_{\left\{r \leq \mu\right\}}\ln \left(1-\widehat{Z}^u_{s-}\right)\widehat{Q}_0(\text{d}s,\text{d}u,\text{d}r)+Mart_t,
\end{multline*}
where $Mart_t$ is the martingale part due to integrals with respect to Brownian motions.

Using that $\E[\widehat{Z}_{s}^2] \leq \E[\widehat{Z}_{s}]$, the following inequalities
$$\sum_{u \in\widehat{\mathbb{G}}_{s}}\left(\widehat{Z}^u_s\right)^2 \geq \frac{1}{ \widehat{N}_t}, \ \sum_{u \in\widehat{\mathbb{G}}_{s\minus}\backslash\{E_{s\minus}\}}- \ln \left(1-\widehat{Z}^u_{s\minus}\right) \geq \sum_{u \in\widehat{\mathbb{G}}_{s\minus}\backslash\{E_{s\minus}\}} \widehat{Z}^u_{s\minus} = 1- \widehat{Z}_{s\minus},$$
and taking the expectation yield
\begin{multline*} 
  \E [ \ln \widehat{Z}_t ]\geq \E[\ln \widehat{Z}_0] + \int_0^t\frac{1}{2}\frac{\sigma^2}{1+\delta^2}\Big(1-2\E[\widehat{Z}_s] +\E \Big[  \frac{1}{ \widehat{N}_t}\Big]\Big)\text{d}s
  \\     +\int_0^t\lambda \E[ \widehat{Z}_{s}]\E[\Theta \ln\Theta + \left(1-\Theta\right)\ln (1-\Theta)]\text{d}s
+ \int_0^t\mu ( 1- \E[\widehat{Z}_s])\text{d}s.
\end{multline*}
Recalling that the law of $\Theta$ is symmetrical, we obtain for any $0 \leq t_0 \leq t$,
\begin{multline*} 
  \E [ \ln \widehat{Z}_t ]\geq \E[\ln \widehat{Z}_{t_0}] + \int_{t_0}^t\Bigg[\Bigg(\frac{1}{2}\frac{\sigma^2}{1+\delta^2}\Big(1+\E \Big[  \frac{1}{ \widehat{N}_t}\Big]\Big)+\mu \Bigg)-\left(\frac{\sigma^2}{1+\delta^2}-2\lambda\E[\Theta \ln\Theta]+\mu\right) \E[\widehat{Z}_s]\Bigg]\text{d}s.
\end{multline*}
We can check that the unique stationary distribution for the random variable $\widehat{N}$ writes
\begin{equation*} 
\P( \widehat{N}_\infty= k) = e^{-\lambda/\mu} \frac{(\lambda/\mu)^{k-1}}{(k-1)!}, \quad \text{for} \ k \geq 1. \end{equation*}
In particular, 
\begin{equation} \label{average1Ninfty} \E \left[  {1}/{ \widehat{N}_t}\right] \underset{t \to \infty}{\Longrightarrow}  \E \left[  {1}/{ \widehat{N}_\infty}\right]=\frac{\mu}{\lambda} \left(1- e^{-\lambda/\mu} \right). \end{equation}
Let $\eps>0$ and $t_0$ be such that 
$$  \E \left[  {1}/{ \widehat{N}_t}\right]\geq \frac{\mu}{\lambda} \left(1- e^{-\lambda/\mu} \right)-\eps, \quad \forall t \geq t_0. $$
Then using Jensen inequality, we obtain that for any $t \geq t_0$,
\begin{multline}\label{min_frac}
 \E [\widehat{Z}_t ]\geq \exp\left\{\E[\ln \widehat{Z}_{t_0}]\right\} \exp\left\{\left[\frac{\sigma^2}{2(1+\delta^2)}\left(1+  \left(1- e^{-\lambda/\mu}\right)\frac{\mu}{\lambda}-\eps\right) + \mu \right](t-t_0)\right\}\\
 \times \exp\left\{-\left(\frac{\sigma^2}{1+\delta^2}-2\lambda\E[\Theta \ln\Theta]+\mu\right) \int_{t_0}^t \E[\widehat{Z}_s]\text{d}s\right\}.   
\end{multline}  
For the sake of readability, let us introduce the two real numbers:
$$ \mathfrak a :=1+  \left(1- e^{-\lambda/\mu}\right)\frac{\mu}{\lambda}-\eps +\frac{2\mu(1+\delta^2)}{\sigma^2} \leq 2- \frac{(1+\delta^2)}{\sigma^2}\left(4\lambda\E[\Theta \ln\Theta]+ 2\mu\right)=: \mathfrak b, $$
as well as
$$ \mathfrak C:= \mathfrak b\exp \left(\E[\ln \widehat{Z}_{t_0}]\right) \quad \text{and} \quad \mathfrak G_t:=  \mathfrak b\int_{t_0}^t\E[\widehat{Z}_s]\text{d}s. $$
Equation \eqref{min_frac} reads
$$ \mathfrak G_t' \geq \mathfrak C\exp\left\{\frac{\sigma^2}{2(1+\delta^2)} \mathfrak a (t-t_0)\right\}\exp\left\{-\frac{\sigma^2}{2(1+\delta^2)}\mathfrak G_t\right\}. $$
Separating variables, we get
$$ \exp\left\{\frac{\sigma^2}{2(1+\delta^2)}\mathfrak G_t\right\} - 1 \geq \frac{\mathfrak C}{\mathfrak a}\left(\exp\left\{\frac{\sigma^2}{2(1+\delta^2)}\mathfrak a (t-t_0)\right\}-1  \right), $$
which finally leads to 
$$ \liminf_{t \to \infty} \frac{1}{t}\int_{t_0}^t\E[\widehat{Z}_s]\text{d}s \geq \frac{1+  \left(1- e^{-\lambda/\mu}\right)\frac{\mu}{\lambda}+\frac{2\mu(1+\delta^2)}{\sigma^2}}{2- \frac{4\lambda(1+\delta^2)}{\sigma^2}\E[\Theta \ln\Theta]+ \frac{2\mu(1+\delta^2)}{\sigma^2}} . $$
The second lower bound is a direct consequence of the Many-to-One formula \eqref{eq: mto}:
$$  \mathbb{E}_{\nu_0}\Bigg[\mathbbm{1}_{\left\{\mathbb{G}(t) \neq \emptyset\right\}}\sum_{u\in\mathbb{G}(t)}\frac{\left(X^u_t\right)^2}{R_t} \Bigg] =  R_0e^{(a-\mu)t} \mathbb{E}_{\nu_0}\Bigg[ \frac{Y^2_t/ \widehat{R}_t}{\widehat{Y}_t} \Bigg]=  R_0e^{(a-\mu)t} \mathbb{E}_{\nu_0}\left[ \widehat{Z}_t\right] $$
and
$$  \mathbb{E}_{\nu_0}\left[\mathbbm{1}_{\left\{\mathbb{G}(t) \neq \emptyset\right\}}\frac{R_t}{N_t} \right] = \mathbb{E}_{\nu_0}\Bigg[\mathbbm{1}_{\left\{\mathbb{G}(t) \neq \emptyset\right\}}\sum_{u\in\mathbb{G}(t)}\frac{X^u_t}{N_t} \Bigg] =  R_0e^{(a-\mu)t} \mathbb{E}_{\nu_0}\left[ \frac{\widehat{Y}_t/ \widehat{N}_t}{\widehat{Y}_t} \right]=  R_0e^{(a-\mu)t} \mathbb{E}_{\nu_0} \left[  \frac{1}{ \widehat{N}_t}\right]. $$
We conclude the proof thanks to the inequality
${R_t}/{N_t} \leq \sum_{u\in\mathbb{G}(t)}(X^u_t)^2/R_t$.
\end{proof}

The next Lemma explores the role of noise in the behavior of the spinal individual. It will appear as a key result in the proof of Proposition \ref{prop: sharing resources}.

\begin{lemma} \label{prop_correl}
 Let $\mu$ and $\lambda$ be fixed. Then 
        $$ \lim_{\sigma^2/(1+\delta^2) \to \infty}\E[\widehat Z_\infty] = 1  .$$
\end{lemma} 
\begin{proof}
    From \eqref{mathfrakX} and \eqref{eq: eds z} we get that
    $$ \lim_{t \to \infty}\Big\{\E\Big[\frac{\sigma^2}{1+\delta^2}\widehat{Z}_t\Big(1-2\widehat{Z}_t +  \sum_{u \in\widehat{\mathbb{G}}_{t}}\Big(\widehat{Z}^u_t\Big)^2\Big)    +  \mu\widehat{Z}_t\sum_{u \in\widehat{\mathbb{G}}_{t}\backslash\{E_t\}}\frac{\widehat{Z}^u_{t}}{1-\widehat{Z}^u_t}-2\widehat{Z}_t^2\lambda K_{\theta}\Big] \Big\}=0.  $$
    Adding \eqref{mathfrakX} and recalling that $K_{\theta} := \mathbb{E}[\Theta(1-\Theta)]$, we deduce that
    \begin{multline}\label{egalite} \lim_{t \to \infty}\E \Big[\frac{\sigma^2\widehat{Z}_t}{1+\delta^2}\sum_{u \in\widehat{\mathbb{G}}_{t}\backslash\{E_t\}}\Big(\widehat{Z}^u_t\Big)^2  +  \mu\widehat{Z}_t\sum_{u \in\widehat{\mathbb{G}}_{t}\backslash\{E_t\}}\frac{\widehat{Z}^u_{t}}{1-\widehat{Z}^u_t}\Big] = \\ \E\Big[ 2\lambda K_{\theta}\widehat{Z}_\infty^2 + \frac{\sigma^2}{1+\delta^2}(2\widehat{Z}_\infty^2-\widehat{Z}_\infty-\widehat{Z}_\infty^3)\Big]. 
    \end{multline}
Noticing that the function $x \mapsto 2x^2-x-x^3$ is non-positive on $[0,1]$ and using that $Z_t^u \in (0,1)$ for any $u \in \widehat{\mathbb{G}}_{t}\backslash\{E_t\}$, we obtain
    $$  \lim_{t \to \infty}\frac{\sigma^2}{1+\delta^2}\E \Big[\widehat{Z}_t\sum_{u \in\widehat{\mathbb{G}}_{t}\backslash\{E_t\}}\Big(\widehat{Z}^u_t\Big)^2 \Big]\leq 2\lambda K_{\theta}\E \Big[ \widehat{Z}_\infty^2\Big], $$
    and thus
\begin{equation}\label{maj_sum}  \lim_{t \to \infty}\E \Big[\widehat{Z}_t\sum_{u \in\widehat{\mathbb{G}}_{t}\backslash\{E_t\}}\left(\widehat{Z}^u_t\right)^2 \Big]\leq (1+\delta^2)\frac{2\lambda K_{\theta}}{\sigma^2}. \end{equation}
 Now using that
 $$ \sum_{u \in\widehat{\mathbb{G}}_{t}\backslash\{E_t\}}\Big(\widehat{Z}^u_t\Big)^2 =(1-\widehat{Z}_t)^2 \sum_{u \in\widehat{\mathbb{G}}_{t}\backslash\{E_t\}}\Big(\frac{\widehat{Z}^u_t}{1- \widehat{Z}_t}\Big)^2  \geq \frac{(1-\widehat{Z}_t)^2}{\widehat{N}_t} $$
 we deduce, using Cauchy-Schwarz inequality
 $$  \lim_{t \to \infty}\Big(\E\Big[\widehat{N}_t^{1/2} \widehat{Z}_t^{1/2} \frac{(1-\widehat{Z}_t)}{\widehat{N}_t^{1/2}} \Big] \Big)^2 \leq  \E\Big[\widehat{N}_\infty\Big]  \lim_{t \to \infty}\E\Big[\widehat{Z}_t \frac{(1-\widehat{Z}_t)^2}{\widehat{N}_t} \Big]\leq \frac{\lambda+\mu}{\lambda} (1+\delta^2) \frac{2\lambda K_{\theta}}{\sigma^2}. $$ 
 At the end, we thus get
 $$ \lim_{\sigma^2/(1+\delta^2) \to \infty} \E\left[\widehat{Z}_\infty^{1/2}(1-\widehat{Z}_\infty) \right]=0$$
 which implies that the random variable $\widehat{Z}_\infty^{1/2}(1-\widehat{Z}_\infty)$ converges in probability to $0$ when $\sigma^2/(1+\delta^2)$ goes to infinity. Adding \eqref{maj_sum} implies that
$$  \lim_{t \to \infty}\E \Big[\sum_{u \in\widehat{\mathbb{G}}_{t}\backslash\{E_t\}}\widehat{Z}^u_t\sum_{u \in\widehat{\mathbb{G}}_{t}\backslash\{E_t\}}\left(\widehat{Z}^u_t\right)^2 \Big]\leq (1+\delta^2)\frac{2\lambda K_{\theta}}{\sigma^2}.$$
Hence 
$ \sum_{u \in\widehat{\mathbb{G}}_{t}\backslash\{E_t\}}(\widehat{Z}^u_t)^2 $
converges in probability to $0$ when $\sigma^2/(1+\delta^2)$ goes to $\infty$, which allows us to conclude that the random variable $1-\widehat{Z}_\infty$ converges in probability to $0$ when $\sigma^2$ goes to infinity. This ends the proof.
\end{proof}

Now we can establish the results presented in Section \ref{section: long-time}.
\subsubsection{Proof of Proposition \ref{prop: basic qts}}
Notice that the process $N_t$ has the same law as an $M/M/\infty$ queue with initial state $1$, until it reaches $0$. The exact law of $N_t$ seems out of reach (we may find its Laplace transform in the literature, but it is not easily invertible), but we know its mean extinction time, which is enough for our purposes. If we denote by $T_0 := \inf \{t \geq 0, N_t=0\}$, we derive from \cite{roijers2007analysis} Equation \eqref{mean_ext_time}
$$ \E_{\Z_0}[T_0]= \frac{1}{\lambda}\sum_{j\geq 1}\left(\frac{\lambda}{\mu}\right)^j\sum_{k=0}^{N_0-1} \frac{k!}{(j+k)!} \leq N_0 \frac{e^{\lambda/\mu}-1}{\lambda}. $$
The inequality is an equality if $\langle \Z_0, 1 \rangle = 1$.
Applying Markov Inequality, we deduce
\begin{equation}\label{eq: P(G)}
    \mathbb{P}\left(\mathbb{G}(t) \neq \emptyset \right) = \P(T_0\geq t)\leq 1 \wedge \frac{\E_{\Z_0}[T_0]}{t} \leq 1 \wedge N_0\frac{e^{\lambda/\mu}-1}{\lambda t}
\end{equation}
which goes to $0$ when $t$ goes to infinity.
A simple lower bound on $\mathbb{P}(\G(t) \neq \emptyset)$ can be derived from the Many-to-One formula \eqref{cor: mto colony}
\begin{equation}\label{eq: mto P survie}
    \mathbb{P}\Big(\mathbb{G}(t) \neq \emptyset \Big) = \E\Big[ \mathbbm{1}_{\G(t) \neq \emptyset}\frac{\sum_{u \in \G(t)}X^u_t}{R_t}\Big] = R_0e^{(a-\mu)t}\E\Big[\frac{1}{\widehat{R}_t}\Big].
\end{equation}
Then, from Jensen inequality we have
$$
\E\left[{1}/{\widehat{R}_t}\right] \geq \exp\left(-\E\left[ \ln \widehat{R}_t \right] \right).
$$
Using Itô formula, we obtain the SDE for $\ln \widehat{R}_t $ and taking the expectation we get
\begin{equation}\label{eq: E ln(r)}
   \frac{\text{d}\E[\ln \widehat{R}_t ]}{\text{d}t} \leq a - \mu + \frac{\sigma^2}{2(1+\delta^2)}\Big(\delta^2 - E\Big[\frac{1}{\widehat{N}_t}\Big] \Big) + \Big(\frac{\sigma^2}{1+\delta^2}+\mu\Big)\E[\widehat{Z}_t].  
\end{equation}

We now use the almost sure bounds $\E[\widehat{Z}_t] \leq 1$ and $E[{1}/{\widehat{N}_t}] \geq 0$. Then \eqref{eq: E ln(r)} becomes
$$
\E[\ln \widehat{R}_t ] \leq \E[\ln\widehat{R}_0 ] + \left(a + \frac{\sigma^2}{1+\delta^2}\left(\frac{\delta^2}{2} +1\right) \right)t.
$$
Notice that the law of the population size $N_t$ only depends on the parameters $\lambda$ and $\mu$. We can thus inject the previous inequality in \eqref{eq: mto P survie} and take $\sigma = 0$. This reads 
\begin{equation}\label{eq: low bound survie}
    \mathbb{P}\left(\mathbb{G}(t) \neq \emptyset \right) \geq e^{-\mu t}.
\end{equation}

From the branching dynamics of the colonial process, we obtain the differential equation
$$
\text{d}\mathbb{E}_{\nu_0}[ N_t] /\text{d}t= \lambda\mathbb{P}\left(\mathbb{G}(t) \neq \emptyset \right) - \mu \mathbb{E}_{\nu_0}\left[ N_t\right].
$$

Using the previously established bounds \eqref{eq: P(G)} and \eqref{eq: low bound survie} in the previous equality we get that
\begin{equation}\label{eq: bounds E N_t}
    \lambda e^{-\mu t}- \mu \mathbb{E}_{\nu_0}\left[ N_t\right] \leq \frac{\text{d}}{\text{d}t}\mathbb{E}_{\nu_0}\left[ N_t\right] \leq \lambda \left(1 \wedge N_0\frac{e^{\lambda/\mu}-1}{\lambda t}\right) - \mu \mathbb{E}_{\nu_0}\left[ N_t\right].
\end{equation}
Integrating the lower bound gives that
$$
 \left(\lambda t +N_0\right)e^{-\mu t} \leq \mathbb{E}_{\nu_0}\left[ N_t\right].
$$
For the upper bound, we integrate on $[0,t_1]\cup(t_1,t]$, where $t_1 := N_0(e^{\lambda/\mu}-1) / \lambda$.
Integrating on $[0,t_1]$, we have that
$$
 \mathbb{E}_{\nu_0}\left[ N_t\right] \leq \frac{\lambda}{\mu}\left(1 -e^{-\mu t}\right) + N_0e^{-\mu t}.
$$
Integrating on $(t_1,t]$, we have that
$$
 \mathbb{E}_{\nu_0}\left[ N_t\right] \leq N_0\left(e^{\frac{\lambda}{\mu}} -1\right) \left( E_i(\mu t) -E_i(\mu t_1)\right)e^{-\mu t}+ n_{t_1}e^{-\mu (t-t-1)}
$$
where $n_{t_1} := (1-e^{-\mu t_1})\lambda / \mu + N_0e^{-\mu t_1}$ and the function $E_i$ is the exponential integral function defined for all $x \in \mathbb{R}_+^*$ as
$$
E_i(x) := \int_{-\infty}^x\frac{e^t}{t}\text{d}t.
$$
We conclude the proof by noticing that $E_i(x)e^{-x} \sim 1/x$ as $x$ goes to infinity \cite{abramowitz68}.

\subsubsection{Proof of Proposition \ref{prop:var}}
First notice that the expression \eqref{eq: basic qts} of the mean resources can be obtained with classical means from the SDEs describing the dynamics of the colonies. However the result can be obtain as a direct application of the Many-to-One formula \eqref{eq: mto} with the function $f(x) = x$.
Now we will establish boundaries on the variance of the total resources in the population.
From the Many-to-One formula \eqref{eq: mto}, with the function  $f(x) = xR$ we have
\begin{equation}\label{eq: distrib lower bound}
     \mathbb{E}_{\nu_0}\left[R^2_t \right] = R_0e^{(a-\mu)t} \mathbb{E}_{\nu_0}\left[ \widehat{R}_t \right].  
\end{equation}
We can thus simply find bounds on $Var[R_t]$ by finding inequalities on the mean resources of the spinal process.
From the SDE \eqref{eq: sde total biomass} we have
\begin{equation}\label{eq: E[R_t hat]}
    \frac{\text{d}\E[ \widehat{R}_t]}{\text{d}t} = \E[\widehat{R}_t]\Big(a - \mu +\frac{\sigma^2\delta^2}{1+\delta^2} \Big) + \E[\widehat{Y}_t]\Big( \mu +\frac{\sigma^2}{1+\delta^2} \Big). 
\end{equation}
We now take the expectation in \eqref{eq: eds indiv spine 2} we have that 
\begin{equation}\label{eq: edo E[Y]}
    \frac{\text{d}\E[ \widehat{Y}_t]}{\text{d}t} = \E\Big[\widehat{Y}_t\Big(a+\sigma^2 - 2\lambda\E[\theta(1-\theta)] \frac{\widehat{Y}_t}{\widehat{R}_t}\Big)\Big]. 
\end{equation}
Then we use, in \eqref{eq: edo E[Y]}, the fact that $0 \leq \widehat{Z}_t \leq 1$ a.s. to obtain the following bounds on $\E[\widehat{Y}_t]$
\begin{equation*}
    \E[\widehat{Y}_t]\left(a+\sigma^2 - 2\lambda\E[\theta(1-\theta)] \right) \leq {\text{d}\E[ \widehat{Y}_t]}/{\text{d}t} \leq \E[\widehat{Y}_t](a+\sigma^2 ). 
\end{equation*}
Recall that $K_{\theta} := \E[\theta(1-\theta)]$.
Integrating these inequalities and injecting into \eqref{eq: E[R_t hat]} yield
$$
 \frac{\text{d}\E_{\Z_0}[ \widehat{R}_t]}{\text{d}t} \leq \E[\widehat{R}_t](a - \mu +\frac{\sigma^2}{1+\delta^2}\delta^2 ) + \Big( \mu +\frac{\sigma^2}{1+\delta^2} \Big)\E[\widehat{Y}_0]e^{( a +\sigma^2 )t},
$$
that integrates into
\begin{equation*}
         \E[\widehat{R}_t] \leq e^{\left( a +\sigma^2 \right)t} + (R_0 -1)e^{\left(a - \mu +\frac{\sigma^2}{1+\delta^2}\delta^2 \right)t}.
\end{equation*}
And
$$
\frac{\text{d}\E_{\Z_0}[ \widehat{R}_t]}{\text{d}t} \geq  \E[\widehat{R}_t]\left(a - \mu +\frac{\sigma^2\delta^2}{1+\delta^2} \right) + \left( \mu +\frac{\sigma^2}{1+\delta^2} \right)\E[\widehat{Y}_0]e^{( a +\sigma^2 - 2\lambda K_{\theta} )t},
$$
that integrates into
\begin{equation*}
         \E[\widehat{R}_t] \geq \frac{\mu +\frac{\sigma^2}{1+\delta^2}}{\mu + \frac{\sigma^2}{1+\delta^2} - 2 \lambda K_{\theta}}\left(e^{\left( a +\sigma^2 -  2 \lambda K_{\theta} \right)t} - e^{\left(a - \mu +\frac{\sigma^2}{1+\delta^2}\delta^2 \right)t} \right) + R_0e^{(a - \mu +\frac{\sigma^2}{1+\delta^2}\delta^2 )t}
\end{equation*}
if $\mu + {\sigma^2}/(1+\delta^2) \neq 2 \lambda K_{\theta}$, and into
$   \E[\widehat{R}_t] \geq (2 \lambda K_{\theta}t + R_0)e^{(a+ \sigma^2 - 2 \lambda K_{\theta} )t}$
if $\mu + {\sigma^2}/(1+\delta^2) = 2 \lambda K_{\theta}.$
We conclude the proof by using \eqref{eq: distrib lower bound} and the definition of the variance.

\subsubsection{Proof of Proposition \ref{prop: sharing resources}}
To obtain the quantity of resources owned by an individual picked uniformly at random in the population, we apply \eqref{eq: sampling} to $F = x^u$. It yields
$$
\E_{\nu_0}\left[\mathbbm{1}_{\left\{\mathbb{G}(t) \neq \emptyset\right\}}X^{U_t}_t\right] = R_0e^{(a-\mu)t}\E_{\nu_0}\left[\frac{1}{\widehat{N}_t}\right] \sim  R_0e^{(a-\mu)t} \frac{\mu}{\lambda}\left(1 - e^{-\lambda/\mu}\right).
$$
From the Many-to-One formula \eqref{eq: mto}, with the function 
$f(x,\nu) := x\mathbbm{1}_{\left\{\frac{x}{\langle \nu, I_d \rangle} \geq \gamma \right\}}$, we get 
\begin{equation}\label{eq: distribution ratios}
  \mathbb{E}_{\nu_0}\Bigg[ \sum_{u \in \mathbb{G}(t)} X^u_t \mathbbm{1}_{\left\{\frac{X^u_t}{R_t} \geq \gamma \right\}} \Bigg] =\E_{\nu_0}[R_t]\ \mathbb{P}_{\nu_0}\left(\frac{\widehat{Y}_t}{\widehat{R}_t} \geq \gamma \right).
\end{equation}
Equation \eqref{tout_ds1} is a direct consequence of point $(1)$ of Lemma \ref{prop_correl} and \eqref{pastout_ds1bis} is a direct consequences of point $(2)$ of Lemma \ref{prop_lim_frac} with $\beta=0$.
Equation \eqref{tout_ds1bis} is a consequence of the Many-to-One formula \eqref{eq: mto}, with the function 
$f(x,\nu) := x^2/R$ and the inequalities
$$ \sum_{u \in \mathbb{G}(t)} \frac{(X^u_t)^2}{R_t}\leq \sum_{u \in \mathbb{G}(t)}\sup_{u \in \mathbb{G}(t)} X^u_t \frac{X^u_t}{R_t}= \sup_{u \in \mathbb{G}(t)} X^u_t,$$
as well as point $(3)$ of Lemma \ref{prop_lim_frac}.

\section*{Acknowledgments.}
We thank J.-F. Delmas and S. C. Harris for their diligent proofreading and constructive comments on this work. We also thank L. Coquille and A. Marguet for the interesting discussions we had on this project. \\
This work is supported by the French National Research Agency in the framework of the "France 2030" program (ANR-15-IDEX-0002) and by the LabEx PERSYVAL-Lab (ANR-11-LABX-0025-01). We acknowledge partial support by the Chair "Modélisation Mathématique et Biodiversité" of VEOLIA-Ecole Polytechnique-MNHN-F.X.

\printbibliography  

\end{document}